\documentclass[pdflatex,sn-mathphys-num]{sn-jnl}% Math and Physical Sciences Numbered Reference Style
\usepackage{graphicx}%
\usepackage{multirow}%
\usepackage{amsmath,amssymb,amsfonts}%
\usepackage{amsthm}%
\usepackage{mathrsfs}%
\usepackage[title]{appendix}%
\usepackage{xcolor}%
\usepackage{textcomp}%
\usepackage{manyfoot}%
\usepackage{booktabs}%
\usepackage{algorithm}%
\usepackage{algorithmicx}%
\usepackage{algpseudocode}%
\usepackage{listings}%
\theoremstyle{thmstyleone}%
\newtheorem{theorem}{Theorem}[section]% meant for sectionwise numbers
\newtheorem{lemma}[theorem]{Lemma}%
\newtheorem{corollary}[theorem]{Corollary}%

\theoremstyle{thmstyletwo}%
\newtheorem{example}{Example}[section]% 添加了 [section]
\newtheorem{remark}{Remark}[section]% 添加了 [section]
\numberwithin{equation}{section}
\theoremstyle{thmstylethree}%
\newtheorem{definition}{Definition}[section]% 添加了 [section]

\begin{document}

\title[$\mathcal F$-Transitivity on Lattice Graphs]
{\(\mathcal F\)-Transitivity of Backward Shifts on Lattice Graphs}

\author*[1]{\fnm{Xiang} \sur{Chen}}\email{2020233018@tju.edu.cn}

\author[2]{\fnm{Cui} \sur{Wang}}\email{cuiwang2016@126.com}

\author[1]{\fnm{Ze-hua} \sur{Zhou}}\email{zehuazhoumath@aliyun.com; zhzhou@tju.edu.cn}

\affil*[1]{\orgdiv{School of Mathematics}, \orgname{Tianjin University},
\orgaddress{\city{Tianjin}, \postcode{300350}, \country{P.R. China}}}

\affil[2]{\orgdiv{School of Mathematics and Statistics},
	\orgname{Henan Normal University},
	\orgaddress{\city{Xinxiang}, \postcode{453007}, \country{P.R. China}}}

\abstract{We study \(\mathcal F\)-transitivity of backward shifts on finite- and
	infinite-width directed lattice graphs. For the unilateral and
	bilateral finite-width lattices on weighted \(\ell^p\)- or
	\(c_0\)-spaces, we obtain exact weight characterizations of
	\(\mathcal F\)- and \(\widetilde{\mathcal F}\)-transitivity for an arbitrary
	Furstenberg family \(\mathcal F\). If \(\mathcal F\) is finitely invariant, these conditions also
	characterize topological \(\mathcal F\)-recurrence. For the infinite quadrant lattice with
	radial weights on weighted \(\ell^2\)-spaces, a finite-dimensional
	reduction and sharp smallest-singular-value estimates for Pascal-type
	transfer matrices yield equivalent characterizations of
	\(\mathcal F\)-transitivity, hypercyclicity and mixing. This provides
	a partial answer, in the radial \(\ell^2\)-setting, to the open problem
	posed by Baranov, Lishanskii and Papathanasiou concerning
	hypercyclicity on the infinite lattice graph. Examples at the critical
	exponential growth rate illustrate the finer dynamical information
	provided by our criteria.}



\maketitle

\section{Introduction}
Linear dynamics studies the orbit structure of continuous linear
operators, with hypercyclicity, weak mixing and mixing among its
central notions; we refer to \cite{CC,BM} for comprehensive accounts
of the subject. Let \(X\) be a separable infinite-dimensional Banach space,
and let \(T\) be a bounded linear operator on \(X\). The operator \(T\) is called
\emph{hypercyclic} if there exists \(x\in X\) such that
\[
\operatorname{Orb}(x,T)
=
\{T^nx:n\in\mathbb N_0\}
\]
is dense in \(X\); such a vector \(x\) is called a
\emph{hypercyclic vector} for \(T\). Weighted shifts constitute one of
the principal classes of examples in linear dynamics, and their
hypercyclic behavior on classical sequence spaces was characterized
by Salas \cite{CRS}. Furstenberg families provide a unified
return-time framework for studying hypercyclicity and stronger
transitivity properties. This approach has been developed for strong
transitivity properties of general operators \cite{BMPP}, for
\(\mathcal F\)-transitive and \(d\mathcal F\)-transitive families of
composition operators on \(L^p\)-spaces \cite{Cx3}, and for
\(\mathcal F\)-transitivity and recurrence of translation semigroups
\cite{HSX}. The linear structure of sets of \(\mathcal F\)-recurrent vectors
has also been studied by Grivaux, L\'opez-Mart\'inez and Peris
\cite{GLP}.

The study of shift operators has subsequently been extended from
classical sequences to spaces indexed by directed trees and more
general directed graphs. Jab{\l}o{\'n}ski, Jung and Stochel
\cite{CKS} established a systematic operator-theoretic framework for
weighted shifts on directed trees. Grosse-Erdmann and Papathanasiou obtained characterizations of
hypercyclicity and mixing \cite{CX}, as well as chaos \cite{KDP},
for weighted shifts on directed trees. Further dynamical
properties, including \(\mathcal F\)-transitivity and topological
\(\mathcal F\)-recurrence, were investigated by Abakumov and Abbar
\cite{AA}, while hypercyclic algebras were studied by Abbar and Costa
Jr.\ \cite{AC}. Related shift problems have also been considered on
the Lipschitz and Hardy spaces of trees \cite{CR,KR}, as well as in
the operator-valued setting of \(\ell^p\)- and \(c_0\)-sums
\cite{MP}. More recently, Baranov, Lishanskii and Papathanasiou
\cite{BLP} initiated the  study of hypercyclic backward
shifts on directed lattice graphs, obtaining complete
characterizations for unilateral and bilateral finite-width lattices
and partial results for the infinite  lattice. They also posed the following open problem:

\medskip
\noindent\textit{Open problem.}
Find necessary and sufficient conditions 
for the backward shift on the infinite lattice \(G_\infty\)
to be hypercyclic on \(\ell^p(V_\infty,\mu)\),
\(1\leq p<\infty\), or \(c_0(V_\infty,\mu)\).
\medskip

The present paper has two main contributions. Our first contribution
concerns the unilateral and bilateral finite-width lattices. For an
arbitrary Furstenberg family \(\mathcal F\), we obtain explicit weight
characterizations of both \(\mathcal F\)- and
\(\widetilde{\mathcal F}\)-transitivity on weighted \(\ell^p\)- and
\(c_0\)-spaces. If \(\mathcal F\) is finitely invariant, the same
conditions also characterize topological \(\mathcal F\)-recurrence.
These results extend the hypercyclicity and weak-mixing
characterizations in \cite{BLP}. The proofs use explicit formulas
for the right inverses constructed in \cite{BLP} and adapt its
triangular and finite-window estimates to the Furstenberg-family
setting. In particular, the bilateral argument controls both
directions at common return times and excludes finitely many
initial times without assuming finite invariance of \(\mathcal F\).
Our second contribution gives a partial answer to the open problem
in \cite{BLP} for the infinite lattice. For radial weights on
\(\ell^2(V_\infty,\mu)\), we reduce the action of \(B^n\) between
diagonal levels to rectangular Pascal-type matrices. A sharp
estimate for their smallest singular values then shows that the
optimal right-inverse norms of the corresponding transfer matrices
are comparable to explicit quantities \(\Phi_q(n)\). Lifting these
finite-dimensional right inverses through the diagonal isometries
\(J_q\), we obtain equivalent weight characterizations of
\(\mathcal F\)-transitivity and, as consequences, exact criteria for
hypercyclicity and mixing in the radial \(\ell^2\)-setting.
A radial counterexample also shows that the necessary condition
in \cite[Proposition~5.1]{BLP} is not sufficient, answering
negatively the question raised after that proposition.

The paper is organized as follows. Section~2 recalls the necessary
background on lattice shifts, Furstenberg families and transitivity.
Section~3 treats the unilateral and bilateral finite-width lattices,
whereas Section~4 develops the finite-dimensional reduction and
applies it to radial shifts on the infinite lattice.
\section{Preliminaries}

\subsection{Lattice graphs and their backward shifts}

Recall that a directed graph is a pair \(G=(V(G),E(G))\), where
\(V(G)\) is its vertex set and
\[
E(G)\subseteq V(G)\times V(G)
\]
is its set of directed edges. For \(v\in V(G)\), let
\[
\operatorname{Chi}(v)
=
\{u\in V(G):(v,u)\in E(G)\}
\]
denote the set of children of \(v\).
Given two directed graphs \(G\) and \(H\), their Cartesian product
\(G\times H\) has vertex set
\[
V(G\times H)=V(G)\times V(H) \quad \text{and} \quad \bigl((u,u'),(v,v')\bigr)\in E(G\times H)
\]
if and only if
\[
u=v\ \text{and}\ (u',v')\in E(H)
\quad\text{or}\quad
u'=v'\ \text{and}\ (u,v)\in E(G).
\]

Following \cite[Section~2]{BLP}, we consider the three directed
lattice graphs
\[
G_m=[1,m]\times\mathbb N,
\qquad
\widetilde G_m=[1,m]\times\mathbb Z,
\qquad
G_\infty=\mathbb N_0\times\mathbb N_0,
\]
where \([1,m]=\{1,\ldots,m\}\). Each factor is endowed with its
natural orientation \(k\to k+1\). We write \(v_{i,j}=(i,j)\).

Let \(G=(V,E)\) be any of these graphs, and let
\(\mu:V\to\mathbb C\setminus\{0\}\) be a weight. For each
\(v_{i,j}\in V\), put
\[
\mu_{i,j}=\mu(v_{i,j}).
\]
We consider the spaces
\[
X=\ell^p(V,\mu),\quad 1\leq p<\infty,
\qquad\text{or}\qquad
X=c_0(V,\mu),
\]
of functions \(f:V\to\mathbb C\), equipped with the usual
weighted norms
\[
\|f\|_{\ell^p(V,\mu)}
=
\biggl(
\sum_{v_{i,j}\in V}
|\mu_{i,j}f(v_{i,j})|^p
\biggr)^{1/p},
\qquad
\|f\|_{c_0(V,\mu)}
=
\sup_{v_{i,j}\in V}|\mu_{i,j}f(v_{i,j})|.
\]

The backward shift associated with \(G\) is defined by
\begin{equation}\label{eq:def-lattice-shift}
	(Bf)(v)
	=
	\sum_{u\in\operatorname{Chi}(v)}f(u),
	\qquad v\in V.
\end{equation}

By the boundedness characterization
\cite[Eq.~(2.2)]{BLP}, written so as to include the boundary vertices,
the operator \(B\) is bounded on \(X\) if and only if there exists
\(C_{\mu}>0\) such that
\begin{equation}\label{eq:lattice-boundedness}
	|\mu_{i,j}|
	\leq
	C_{\mu}
	\min\left\{
	|\mu_{k,l}|:
	v_{k,l}\in\operatorname{Chi}(v_{i,j})
	\right\},
	\qquad v_{i,j}\in V.
\end{equation}
Throughout this paper, all backward shifts under consideration are
assumed to be bounded.
\subsection{Furstenberg families and transitivity}
\label{subsec:Furstenberg-families}
We first recall the notions concerning Furstenberg families that will
be used throughout the paper.
\begin{definition}
	A nonempty collection \(\mathcal F\) of subsets of
	\(\mathbb N_0\) is called a \emph{Furstenberg family} if the following
	conditions hold:
	\begin{enumerate}
		\item[\textnormal{(i)}]
		every member of \(\mathcal F\) is infinite;
		
		\item[\textnormal{(ii)}]
		if \(A\in\mathcal F\) and
		\(A\subseteq C\subseteq\mathbb N_0\), then \(C\in\mathcal F\).
	\end{enumerate}
\end{definition}

Let \(\mathcal F\) be a Furstenberg family. Following \cite{BMPP}, we
denote by \(\widetilde{\mathcal F}\) the enlarged Furstenberg family consisting
of all sets \(A\subseteq\mathbb N_0\) such that, for every
\(N\in\mathbb N_0\), there exists \(D\in\mathcal F\) satisfying
\[
\bigl(D+[-N,N]_{\mathbb Z}\bigr)\cap\mathbb N_0
\subseteq A,
\]
where
\[
[-N,N]_{\mathbb Z}
=
\{-N,-N+1,\ldots,N-1,N\}
\]
and
\[
D+[-N,N]_{\mathbb Z}
=
\{m+n:m\in D,\ n\in[-N,N]_{\mathbb Z}\}.
\]
Taking \(N=0\) shows that
\(
\widetilde{\mathcal F}\subseteq\mathcal F.
\)
Moreover, every member of \(\widetilde{\mathcal F}\) is thick, that is, it
contains intervals of arbitrarily large finite length.

Let \(X\) be a separable infinite-dimensional Banach space and let \(T\) be a bounded linear operator on \(X\). For nonempty open sets \(U,V\subseteq X\), the
return time set from \(U\) to \(V\) is defined by
\[
N_T(U,V)
=
\{n\in\mathbb N_0:T^n(U)\cap V\neq\phi\}.
\]

\begin{definition}
	Let \(\mathcal F\) be a Furstenberg family on \(\mathbb N_0\).
	\begin{enumerate}
		\item[\textnormal{(i)}]
		The operator \(T\) is called
		\emph{\(\mathcal F\)-transitive} if
		\(
		N_T(U,V)\in\mathcal F
		\)
		for every pair of nonempty open sets \(U,V\subseteq X\).
		
		\item[\textnormal{(ii)}]
		The operator \(T\) is called
		\emph{topologically \(\mathcal F\)-recurrent} if
		\(
		N_T(U,U)\in\mathcal F
		\)
		for every nonempty open set \(U\subseteq X\).
	\end{enumerate}
\end{definition}

\begin{definition}
	Let \(T\in\mathcal L(X)\).
	\begin{enumerate}
		\item[\textnormal{(i)}]
		The operator \(T\) is called \emph{weakly mixing} if
		\(T\oplus T\) is hypercyclic on \(X\oplus X\), where
		\[
		(T\oplus T)(x,y)=(Tx,Ty),
		\qquad (x,y)\in X\oplus X.
		\]
		
		\item[\textnormal{(ii)}]
		The operator \(T\) is called \emph{mixing} if, for every pair of
		nonempty open sets \(U,V\subseteq X\), there exists
		\(n_0\in\mathbb N_0\) such that
		\[
		T^n(U)\cap V\neq\phi,
		\qquad n\geq n_0.
		\]
	\end{enumerate}
\end{definition}

\begin{remark}\label{rem:classical-Furstenberg-families}
	Consider the following standard Furstenberg families:
	\[
	\begin{aligned}
		\mathcal F_\infty
		&=
		\{A\subseteq\mathbb N_0:A\text{ is infinite}\},\\
		\mathcal F_{\mathrm{thick}}
		&=
		\{A\subseteq\mathbb N_0:
		A\text{ contains arbitrarily long intervals}\},\\
		\mathcal F_{\mathrm{cof}}
		&=
		\{A\subseteq\mathbb N_0:
		\mathbb N_0\setminus A\text{ is finite}\}.
	\end{aligned}
	\]
	For every continuous linear operator on a separable
	infinite-dimensional Banach space,
	\[
	\begin{aligned}
		T\text{ is hypercyclic}
		&\quad\Longleftrightarrow\quad
		T\text{ is }\mathcal F_\infty\text{-transitive},\\
		T\text{ is weakly mixing}
		&\quad\Longleftrightarrow\quad
		T\text{ is }\mathcal F_{\mathrm{thick}}\text{-transitive},\\
		T\text{ is mixing}
		&\quad\Longleftrightarrow\quad
		T\text{ is }\mathcal F_{\mathrm{cof}}\text{-transitive}.
	\end{aligned}
	\]
	See, for instance, \cite{BMPP}.
\end{remark}

\section{\(\mathcal F\)-transitivity on finite-width lattice graphs}
In \cite[Theorems~2.1 and~2.2]{BLP}, hypercyclicity and weak mixing
were characterized for unilateral and bilateral finite-width lattice
graphs. We extend these results to \(\mathcal F\)-transitivity for an
arbitrary Furstenberg family.

\subsection{The unilateral lattice}

Let \(m\in\mathbb N\). Following \cite{BLP}, consider the directed
lattice graph
\[
G_m=(V_m,E_m),
\qquad
V_m=\{v_{i,j}:1\leq i\leq m,\ j\in\mathbb N\},
\]
with edge set
\[
\begin{aligned}
	E_m
	={}&
	\bigl\{(v_{i,j},v_{i,j+1}):
	1\leq i\leq m,\ j\in\mathbb N\bigr\}\\
	&\cup
	\bigl\{(v_{i,j},v_{i+1,j}):
	1\leq i<m,\ j\in\mathbb N\bigr\}.
\end{aligned}
\]
The case \(m=3\) is illustrated in
Figure~\ref{fig:finite-width-lattice}.

\begin{figure}[htbp]
	\centering
	\includegraphics[width=0.5\textwidth]{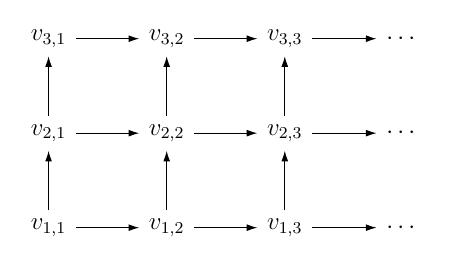}
	\caption{The directed lattice graph \(G_m\) for \(m=3\).}
	\label{fig:finite-width-lattice}
\end{figure}

Write \(e_{i,j}=e_{v_{i,j}}\) for \(v_{i,j}\in V_m\), and put
\(D=c_{00}(V_m)\). For integer indices, set \(e_{i,j}=0\)
whenever \(i\notin\{1,\ldots,m\}\) or \(j\leq0\), and use
the convention that \(\binom nr=0\) for \(r>n\).
By \cite[Eq.~(3.1)]{BLP}, for \(1\leq i\leq m\),
\(j\in\mathbb N\), and \(n\in\mathbb N\),
\begin{equation}\label{eq:iterate-basis}
	B^ne_{i,j}
	=
	\sum_{r=0}^{i-1}
	\binom nr e_{i-r,j-n+r}.
\end{equation}
Thus terms with \(r>n\) or \(j-n+r\leq0\) vanish.

For later use, we record an explicit form of the right inverses
constructed recursively in \cite[Eqs.~(3.2)--(3.4)]{BLP}. For
\(n\in\mathbb N\), define \(R_n:D\to D\) by
\begin{equation}\label{eq:right-inverse-definition}
	R_ne_{i,j}
	=
	\sum_{r=0}^{i-1}
	(-1)^r
	\binom{n+r-1}{r}
	e_{i-r,j+n+r}.
\end{equation}
Then
\begin{equation}\label{eq:right-inverse}
	B^nR_ng=g,
	\qquad g\in D.
\end{equation}
Indeed, for \(0\leq q\leq i-1\), the coefficient of
\(e_{i-q,j+q}\) in \(B^nR_ne_{i,j}\) is
\[
\sum_{r=0}^{q}
(-1)^r
\binom{n+r-1}{r}
\binom n{q-r}.
\]
This is the coefficient of \(z^q\) in the formal product
\[
(1+z)^{-n}(1+z)^n=1.
\]
It is therefore equal to \(1\) for \(q=0\) and to \(0\) for
\(q\geq1\), proving \eqref{eq:right-inverse}.

Moreover, for every \(n\geq1\), it follows that
\begin{equation}\label{eq:binomial-bound}
	\binom{n+r-1}{r}
	\leq
n^r,
	\qquad
	n\geq1,\quad 0\leq r\leq m-1.
\end{equation}

\begin{theorem}\label{thm:F-transitivity-Gm}
	Let \(B\) be a bounded backward shift on
	\(X=\ell^p(V_m,\mu)\), \(1\leq p<\infty\), or on
	\(X=c_0(V_m,\mu)\). For every nonempty finite set
	\(E\subseteq V_m\) and every \(\varepsilon>0\), define
	\[
	\mathcal A(E,\varepsilon)
	=
	\bigcap_{v_{i,j}\in E}
	\left\{
	n\in\mathbb N:
	n^{m-i}|\mu_{i,j+n}|<\varepsilon
	\right\}.
	\]
	Let \(\mathcal F\) be a Furstenberg family on \(\mathbb N_0\). Then the following assertions are equivalent:
	\begin{enumerate}
		\item[\textnormal{(i)}]
		\(B\) is \(\widetilde{\mathcal F}\)-transitive;
		
		\item[\textnormal{(ii)}]
		\(B\) is \(\mathcal F\)-transitive;
		
		\item[\textnormal{(iii)}]
		\(\mathcal A(E,\varepsilon)\in\mathcal F\) for every
		nonempty finite set \(E\subseteq V_m\) and every
		\(\varepsilon>0\).
	\end{enumerate}
	
	Suppose, in addition, that
	\[
	A\in\mathcal F
	\quad\Longrightarrow\quad
	A\cap[N,\infty)\in\mathcal F
	\]
	for every \(N\in\mathbb N\). Then the preceding assertions are
	also equivalent to
	\begin{enumerate}
		\item[\textnormal{(iv)}]
		\(B\) is topologically \(\mathcal F\)-recurrent.
	\end{enumerate}
\end{theorem}
	
	\begin{proof}	
			The implication
		\(\mathrm{(i)}\Rightarrow\mathrm{(ii)}\)
		follows immediately from
		\(
		\widetilde{\mathcal F}\subseteq\mathcal F.
		\)
		
We next prove
\(\mathrm{(ii)}\Rightarrow\mathrm{(iii)}\).
Suppose that \(B\) is \(\mathcal F\)-transitive. Fix a nonempty
finite set \(E\subseteq V_m\) and let \(\eta>0\). Put
\[
K=\max_{v_{i,j}\in E}(j+m-i),
\qquad
y=\sum_{r=1}^{K}e_{m,r},
\]
and
\[
\mu_0
=
\min\left\{
|\mu_{i,r}|:
1\leq i\leq m,\quad 1\leq r\leq K
\right\}>0.
\]
Set
\[
C_m=2^{m-1},
\qquad
a_m=\frac1{(m-1)!},
\qquad
c_m=\frac{a_m}{2}.
\]
Choose \(0<\delta<1\) sufficiently small that
\begin{equation}\label{eq:delta-choice}
	C_m\delta\leq\frac{a_m}{2},
	\qquad
	\delta\mu_0<\|y\|_X.
\end{equation}
Let
\[
n_0=\left\lceil\frac m\delta\right\rceil,
\qquad
M_\mu=\max\{1,C_\mu\}^{m-1},
\]
where \(\lceil t\rceil\) denotes the smallest integer greater
than or equal to \(t\).
Choose \(\varepsilon>0\) such that
\begin{equation}\label{eq:epsilon-choice}
	\frac{M_\mu\varepsilon}{c_m}<\eta
\end{equation}
and
\begin{equation}\label{eq:small-times}
	\max_{0\leq q<n_0}\|B^q\|\varepsilon
	<
	\|y\|_X-\delta\mu_0.
\end{equation}
Define
\[
U=B_X(0,\varepsilon),
\qquad
V=B_X(y,\delta\mu_0).
\]
Condition \eqref{eq:small-times} ensures that
\(B^q(U)\cap V=\varnothing\) for \(0\leq q<n_0\). Hence
\[
N_B(U,V)\subseteq[n_0,\infty),
\]
while \(\mathcal F\)-transitivity gives
\(N_B(U,V)\in\mathcal F\).

Fix \(n\in N_B(U,V)\), and choose \(h\in U\) such that
\(B^nh\in V\). Then
\[
\|h\|_X<\varepsilon,
\qquad
\|B^nh-y\|_X<\delta\mu_0.
\]
Consequently,
\[
|(B^nh-y)(v_{i,r})|
\leq
\frac{\|B^nh-y\|_X}{|\mu_{i,r}|}
<\delta,
\qquad
1\leq i\leq m,\quad 1\leq r\leq K.
\]

We now estimate the coordinates of \(h\) successively.
For \(l=0\), the last inequality gives
\[
|h(v_{m,r+n})-1|
=
|(B^nh-y)(v_{m,r})|
<\delta,
\qquad 1\leq r\leq K.
\]
Suppose that \(1\leq l\leq m-1\) and that
\[
\left|
h(v_{m-s,t+n})-\frac{(-1)^s}{s!}n^s
\right|
\leq 2^s\delta n^s
\]
for \(0\leq s<l\) and \(s+1\leq t\leq K\).
For \(l+1\leq r\leq K\), the formula for \(B^n\) yields
\[
h(v_{m-l,r+n})
=
(B^nh)(v_{m-l,r})
-
\sum_{k=1}^{l}
\binom nk h(v_{m-l+k,r+n-k}),
\]
where \(|(B^nh)(v_{m-l,r})|<\delta\).
The induction hypothesis applies to each term in the sum,
since
\[
l-k+1\leq r-k\leq K.
\]

Since \(n\geq n_0\geq m\), the product formula for the binomial
coefficients gives
\[
0\leq
\frac1{k!}-\frac{\binom nk}{n^k}
\leq
\frac{k(k-1)}{2k!\,n}
\leq\frac1n,
\qquad 1\leq k\leq m-1.
\]
Using these bounds, the induction hypothesis, and
\[
\sum_{k=0}^{l}
\frac{(-1)^{l-k}}{k!(l-k)!}=0,
\]
we obtain
\[
\begin{aligned}
	\frac1{n^l}
	\left|
	h(v_{m-l,r+n})-\frac{(-1)^l}{l!}n^l
	\right|
	\leq
	\frac{\delta}{n^l}
	+
	\delta\sum_{k=1}^{l}2^{l-k}
	+
	\frac ln<
	2^l\delta.
\end{aligned}
\]
Here we used \(l+\delta\leq m\) and \(m/n\leq\delta\).
Thus the induction gives
\begin{equation*}
	\left|
	h(v_{m-l,r+n})-\frac{(-1)^l}{l!}n^l
	\right|
	\leq
	2^l\delta n^l
	\leq C_m\delta n^l
\end{equation*}
for \(0\leq l\leq m-1\) and \(l+1\leq r\leq K\).
By \eqref{eq:delta-choice},
\begin{equation}\label{eq:triangular-estimate}
	|h(v_{m-l,r+n})|
	\geq
	\left(\frac1{l!}-C_m\delta\right)n^l
	\geq c_mn^l.
\end{equation}

Now fix \(v_{i,j}\in E\), and put
\[
l=m-i,
\qquad r=j+l.
\]
Then for  \(l+1\leq r\leq K\), and \eqref{eq:triangular-estimate}
 yields
\[
|h(v_{i,j+l+n})|\geq c_mn^l.
\]
Since
\[
|h(v_{i,j+l+n})\mu_{i,j+l+n}|
\leq\|h\|_X<\varepsilon,
\]
it follows that
\[
n^l|\mu_{i,j+l+n}|<\frac{\varepsilon}{c_m}.
\]
Applying \eqref{eq:lattice-boundedness} along \(l\) horizontal
edges gives
\[
|\mu_{i,j+n}|
\leq C_\mu^l|\mu_{i,j+l+n}|
\leq M_\mu|\mu_{i,j+l+n}|.
\]
It follows from \eqref{eq:epsilon-choice} that
\[
n^{m-i}|\mu_{i,j+n}|
\leq
M_\mu n^l|\mu_{i,j+l+n}|
<
\frac{M_\mu\varepsilon}{c_m}
<
\eta.
\]
This holds for every \(v_{i,j}\in E\), proving
\begin{equation}\label{eq:return-set-inclusion}
	N_B(U,V)\subseteq\mathcal A(E,\eta).
\end{equation}
Since \(N_B(U,V)\in\mathcal F\) and \(\mathcal F\) is upward
hereditary, \(\mathcal A(E,\eta)\in\mathcal F\).
Therefore \(\mathrm{(iii)}\) holds.
		
		We now prove
		\(\mathrm{(iii)}\Rightarrow\mathrm{(ii)}\).
		Assume \(\mathrm{(iii)}\), and let \(U,V\subseteq X\) be
		nonempty open sets. Since \(D\) is dense in \(X\), choose
		\[
		f\in U\cap D,
		\qquad
		0\ne g\in V\cap D,
		\]
		and \(\rho>0\) such that
		\[
		B_X(f,\rho)\subseteq U,
		\qquad
		B_X(g,\rho)\subseteq V.
		\]
		Since \(f\) has finite support, there exists \(n_0\geq2\) such that
		\begin{equation}\label{eq:finite-support-vanishing}
			B^nf=0,
			\qquad n\geq n_0.
		\end{equation}
		
		Write
		\[
		g=\sum_{v_{i,j}\in F}g_{i,j}e_{i,j},
		\qquad F=\operatorname{supp}(g),
		\]
		and define
		\[
		F^\sharp
		=
		\left\{
		v_{i-r,j+r}:
		v_{i,j}\in F,\quad 0\leq r\leq i-1
		\right\}.
		\]
		Set
		\[
		C_g=\sum_{v_{i,j}\in F}i|g_{i,j}|>0.
		\]
		Since \(F^\sharp\) is nonempty and finite and all weights
		are nonzero, we also have
		\[
		d=
		\min\left\{
		n^{m-s}|\mu_{s,t+n}|:
		v_{s,t}\in F^\sharp,\quad 1\leq n<n_0
		\right\}>0.
		\]
		Choose
		\[
		0<\eta<\min\left\{\frac{\rho}{C_g},d\right\}.
		\]
		For \(1\leq n<n_0\) and every \(v_{s,t}\in F^\sharp\),
		\[
		n^{m-s}|\mu_{s,t+n}|\geq d>\eta,
		\]
		so \(n\notin\mathcal A(F^\sharp,\eta)\). Hence
		\begin{equation}\label{eq:good-times-large}
			\mathcal A(F^\sharp,\eta)\subseteq[n_0,\infty).
		\end{equation}
		By \(\mathrm{(iii)}\),
		\[
		\mathcal A(F^\sharp,\eta)\in\mathcal F.
		\]
		
		Fix \(n\in\mathcal A(F^\sharp,\eta)\).
		For each \(v_{i,j}\in F\),
		\eqref{eq:right-inverse-definition} and
		\eqref{eq:binomial-bound} give
		\[
		\begin{aligned}
			\|R_ne_{i,j}\|_X
			\leq
			\sum_{r=0}^{i-1}
			\binom{n+r-1}{r}|\mu_{i-r,j+n+r}|
			\leq
			\sum_{r=0}^{i-1}
			n^{m-(i-r)}|\mu_{i-r,j+n+r}|
			< i\eta.
		\end{aligned}
		\]
		Here we used \(n\geq1\), \(r\leq m-(i-r)\), and
		\(v_{i-r,j+r}\in F^\sharp\).
		Therefore,
		\begin{equation}\label{eq:right-inverse-small}
			\|R_ng\|_X
			\leq
			\sum_{v_{i,j}\in F}|g_{i,j}|\|R_ne_{i,j}\|_X
			\leq C_g\eta<\rho.
		\end{equation}
		
		Set \(x_n=f+R_ng\). Then
		\[
		\|x_n-f\|_X=\|R_ng\|_X<\rho,
		\]
		so \(x_n\in U\). Moreover, since \(n\geq n_0\),
		\eqref{eq:right-inverse} and
		\eqref{eq:finite-support-vanishing} yield
		\[
		B^nx_n=B^nf+B^nR_ng=g\in V.
		\]
		Thus
		\[
		\mathcal A(F^\sharp,\eta)\subseteq N_B(U,V).
		\]
		Since \(\mathcal A(F^\sharp,\eta)\in\mathcal F\) and
		\(\mathcal F\) is upward hereditary,
		\(N_B(U,V)\in\mathcal F\).
		Therefore \(B\) is \(\mathcal F\)-transitive.
		
		It remains to prove
		\(\mathrm{(ii)}\Rightarrow\mathrm{(i)}\).
		If \(B\) is \(\mathcal F\)-transitive, then \(B\) is topologically
		transitive and hence hypercyclic. By the hypercyclicity
		characterization for \(G_m\), hypercyclicity and weak mixing are
		equivalent; see \cite[Theorem~2.1]{BLP}. Hence \(B\) is weakly mixing.
		By \cite[Lemma~2.3]{BMPP}, an \(\mathcal F\)-transitive
		and weakly mixing operator is \(\widetilde{\mathcal F}\)-transitive.
		
	Finally, under the additional tail-stability assumption,
	\(\mathrm{(ii)}\Rightarrow\mathrm{(iv)}\) is immediate. Conversely,
	suppose that \(B\) is topologically \(\mathcal F\)-recurrent. Let \(U,V\subseteq X\) be nonempty
		open sets. Choose \(z_0\in V\). Since \(D\) is dense in \(X\),
		we may choose
		\[
		h\in D\cap(U-z_0).
		\]
		Then
		\[
		V_0=V\cap(U-h)
		\]
		is a nonempty open neighborhood of \(z_0\). In particular,
		\[
		V_0\subseteq V,
		\qquad
		h+V_0\subseteq U.
		\]
		
		Since \(h\) has finite support, there exists \(N\in\mathbb N\)
		such that
		\[
		B^nh=0,
		\qquad n\geq N.
		\]
		By topological \(\mathcal F\)-recurrence,
		\[
		N_B(V_0,V_0)\in\mathcal F.
		\]
		Since \(\mathcal F\) is stable under finite deletions,
		\[
		A
		=
		N_B(V_0,V_0)\cap[N,\infty)
		\in\mathcal F.
		\]
		
		For every \(n\in A\), there exists \(z_n\in V_0\) such that
		\(B^nz_n\in V_0\). Hence
		\[
		h+z_n\in U
		\]
		and
		\[
		B^n(h+z_n)
		=
		B^nh+B^nz_n
		=
		B^nz_n
		\in V_0
		\subseteq V.
		\]
		Thus
		\[
		A\subseteq N_B(U,V).
		\]
		Since \(\mathcal F\) is upward hereditary, it follows that
		\(N_B(U,V)\in\mathcal F\). Therefore \(B\) is
		\(\mathcal F\)-transitive.
	\end{proof}	
	\subsection{The bilateral lattice}
	Let
	\[
	\widetilde V_m
	=
	\{v_{i,j}:1\le i\le m,\ j\in\mathbb Z\},
	\]
	and consider the directed lattice graph
	\(\widetilde G_m=[1,m]\times\mathbb Z\) whose edges are
	\[
	v_{i,j}\longrightarrow v_{i,j+1},
	\qquad
	1\le i\le m,\quad j\in\mathbb Z,
	\]
	and
	\[
	v_{i,j}\longrightarrow v_{i+1,j},
	\qquad
	1\le i<m,\quad j\in\mathbb Z.
	\]
	Let \(X=\ell^p(\widetilde V_m,\mu)\), \(1\le p<\infty\), or
	\(X=c_0(\widetilde V_m,\mu)\), and suppose that the backward shift
	\(B\) is bounded on \(X\). We write \(e_{i,j}=e_{v_{i,j}}\).
	
For \(1\leq i\leq m\) and \(j\in\mathbb Z\), write
\(e_{i,j}=e_{v_{i,j}}\), and set \(e_{i,j}=0\) whenever
\(i\notin\{1,\ldots,m\}\). Put
\[
D_{\mathbb Z}
=
c_{00}(\widetilde V_m)
=
\operatorname{span}
\{e_{i,j}:1\leq i\leq m,\ j\in\mathbb Z\}.
\]
With this convention, formulas \eqref{eq:iterate-basis} and
\eqref{eq:right-inverse-definition} remain valid for
\(j\in\mathbb Z\), where \(\binom nr=0\) for \(r>n\).
For each \(n\in\mathbb N\), define \(R_n\) on the coordinate
vectors by \eqref{eq:right-inverse-definition} and extend it
linearly to \(D_{\mathbb Z}\). Each image has finite support
contained in \(\widetilde V_m\), so
\(R_n:D_{\mathbb Z}\to D_{\mathbb Z}\) is well defined.
The same computation as in the unilateral case yields
\[
B^nR_nf=f,
\qquad n\in\mathbb N,\quad f\in D_{\mathbb Z}.
\]
	
	We can now give the Furstenberg-family characterization.
	
	\begin{theorem}\label{thm:F-transitivity-bilateral-lattice}
		Let \(\widetilde G_m=[1,m]\times\mathbb Z\), let
		\(\mu=(\mu_{i,j})\) be a non-zero weight, and let
		\[
		X=\ell^p(\widetilde V_m,\mu),\qquad 1\le p<\infty,
		\]
		or \(X=c_0(\widetilde V_m,\mu)\). Suppose that the backward shift
		\(B\) is bounded on \(X\), and let \(\mathcal F\) be a Furstenberg
		family on \(\mathbb N_0\). Then the following assertions are
		equivalent:
		\begin{enumerate}
			\item[\textnormal{(i)}]
			\(B\) is \(\widetilde{\mathcal F}\)-transitive;
			
			\item[\textnormal{(ii)}]
			\(B\) is \(\mathcal F\)-transitive;
			
			\item[\textnormal{(iii)}]
			for every non-empty finite set
			\(E\subseteq\widetilde V_m\) and every \(\varepsilon>0\),
			\[
			\mathcal A(E,\varepsilon)
			=
			\bigcap_{v_{i,j}\in E}
			\left\{
			n\in\mathbb N:
			n^{m-i}
			\bigl(
			|\mu_{i,j+n}|+|\mu_{i,j-n}|
			\bigr)
			<\varepsilon
			\right\}
			\in\mathcal F.
			\]
		\end{enumerate}
		If, in addition,
		\begin{equation}\label{eq:tail-condition-bilateral}
			A\in\mathcal F
			\quad\Longrightarrow\quad
			A\cap[N,\infty)\in\mathcal F
			\qquad(N\in\mathbb N),
		\end{equation}
		then the preceding assertions are also equivalent to
		\begin{enumerate}
			\item[\textnormal{(iv)}]
			\(B\) is topologically \(\mathcal F\)-recurrent.
		\end{enumerate}
	\end{theorem}
\begin{proof}
	The equivalence
	\(\mathrm{(i)}\Longleftrightarrow\mathrm{(ii)}\)
	follows from the same argument as in the proof of
	Theorem~\ref{thm:F-transitivity-Gm}, using
	\cite[Theorem~2.2]{BLP} in place of
	\cite[Theorem~2.1]{BLP}.
	
We now prove
\(\mathrm{(iii)}\Rightarrow\mathrm{(ii)}\).
Assume \(\mathrm{(iii)}\), and let \(U,V\subseteq X\) be
nonempty open sets. Since \(D_{\mathbb Z}\) is dense in \(X\),
choose nonzero vectors
\[
f\in U\cap D_{\mathbb Z},
\qquad
g\in V\cap D_{\mathbb Z},
\]
and \(\rho>0\) such that
\[
B_X(f,\rho)\subseteq U,
\qquad
B_X(g,\rho)\subseteq V.
\]

For \(h\in\{f,g\}\), define
\[
E_h
=
\left\{
v_{i-r,j+r}:
v_{i,j}\in\operatorname{supp}(h),\quad
0\leq r\leq i-1
\right\},
\]
and set
\[
C_h
=
\sum_{v_{i,j}\in\operatorname{supp}(h)}
i|h(v_{i,j})|>0.
\]
Put \(E=E_f\cup E_g\), and choose
\[
0<\theta<
\min\left\{\frac{\rho}{C_f},\frac{\rho}{C_g}\right\}.
\]
Since \(E\) is nonempty and finite, condition \(\mathrm{(iii)}\)
gives
\[
\mathcal A(E,\theta)\in\mathcal F.
\]

Fix \(n\in \mathcal A(E,\theta)\). For each
\(v_{i,j}\in\operatorname{supp}(f)\cup\operatorname{supp}(g)\),
\eqref{eq:iterate-basis},
\eqref{eq:right-inverse-definition}, and
\eqref{eq:binomial-bound} yield
\[
\begin{aligned}
	\|B^ne_{i,j}\|_X+\|R_ne_{i,j}\|_X
	&\leq
	\sum_{r=0}^{i-1}
	n^r
	\bigl(
	|\mu_{i-r,j+r-n}|+|\mu_{i-r,j+r+n}|
	\bigr)\\
	&\leq
	\sum_{r=0}^{i-1}
	n^{m-i+r}
	\bigl(
	|\mu_{i-r,j+r-n}|+|\mu_{i-r,j+r+n}|
	\bigr)\\
	&\leq i\theta.
\end{aligned}
\]
Here we used
\(\binom nr\leq n^r\), \(n\geq1\),
\(r\leq m-i+r\), and \(v_{i-r,j+r}\in E\).
Consequently, for \(h\in\{f,g\}\),
\begin{equation}\label{eq:bilateral-finite-estimate}
	\begin{aligned}
		\|B^nh\|_X+\|R_nh\|_X
		\leq
		\sum_{v_{i,j}\in\operatorname{supp}(h)}
		|h(v_{i,j})|
		\bigl(
		\|B^ne_{i,j}\|_X+\|R_ne_{i,j}\|_X
		\bigr)<\rho.
	\end{aligned}
\end{equation}

Set \(x_n=f+R_ng\). Then
\[
\|x_n-f\|_X=\|R_ng\|_X<\rho,
\]
so \(x_n\in U\). Moreover,
\eqref{eq:right-inverse} gives
\[
\|B^nx_n-g\|_X
=
\|B^nf\|_X
<
\rho,
\]
and hence \(B^nx_n\in V\). Thus
\[
\mathcal A(E,\theta)\subseteq N_B(U,V).
\]
Since \(\mathcal A(E,\theta)\in\mathcal F\) and \(\mathcal F\) is upward hereditary,
\(N_B(U,V)\in\mathcal F\).
Therefore \(B\) is \(\mathcal F\)-transitive.

		We first establish a finite-window estimate. Fix a nonempty
		finite set \(E\subseteq\widetilde V_m\), \(\varepsilon>0\),
		and nonzero scalars \(a,b\). Put
		\[
		J=\max_{v_{i,j}\in E}|j|,
		\]
		choose an integer \(K>J+m\), and set
		\[
		g_K=\sum_{j=-K}^{K}e_{m,j},
		\qquad
		\mu_K=\min_{\substack{1\leq i\leq m\\ |j|\leq K}}
		|\mu_{i,j}|>0.
		\]
		Define
		\[
		\beta=\max\{|a|,|b|\},
		\qquad
		c_1=\frac{|b|}{2(m-1)!},
		\qquad
		c_2=\frac{|a|}{2(m-1)!}.
		\]
		Choose
		\[
		0<\delta<2^{-m}\min\{c_1,c_2\},
		\]
	put
	\[
	n_0 = \max\left\{ K+J+1,\, \left\lceil\frac{m\beta}{\delta}\right\rceil \right\}
	\quad \text{and} \quad
	0 < \rho < \min\left\{ \delta\mu_K,\, \frac{\varepsilon}{c_1^{-1}+c_2^{-1}} \right\}.
	\]
		
	We follow the induction used in the proof of
	\(\mathrm{(ii)}\Rightarrow\mathrm{(iii)}\) in
	Theorem~\ref{thm:F-transitivity-Gm} for the positive direction,
	and obtain an additional estimate for the negative direction.	Fix
		\[
		n\in
		N_B\bigl(B_X(ag_K,\rho),B_X(bg_K,\rho)\bigr)
		\cap[n_0,\infty),
		\]
		and choose \(f\in X\) such that
		\[
		\|f-ag_K\|_X<\rho,
		\qquad
		\|B^nf-bg_K\|_X<\rho.
		\]
		Then
		\begin{equation}\label{eq:window-pointwise}
			|f(v_{i,j})-ag_K(v_{i,j})|<\delta,
			\qquad
			|(B^nf)(v_{i,j})-bg_K(v_{i,j})|<\delta
		\end{equation}
		for \(1\leq i\leq m\) and \(|j|\leq K\).
		
		Since \(n\geq n_0\geq m\), we have
		\[
		0\leq\frac{\binom nk}{n^k}\leq\frac1{k!}\leq1,
		\qquad
		0\leq\frac1{k!}-\frac{\binom nk}{n^k}\leq\frac1n,
		\qquad 0\leq k\leq m-1.
		\]
		For the positive direction, \eqref{eq:window-pointwise} gives
		\[
		|f(v_{m,j+n})-b|<\delta,
		\qquad |j|\leq K.
		\]
		Suppose that
		\[
		\left|
		f(v_{m-s,t+n})-b\frac{(-1)^s}{s!}n^s
		\right|
		\leq2^s\delta n^s
		\]
		holds for \(0\leq s<\ell\) and \(-K+s\leq t\leq K\).
		For \(-K+\ell\leq j\leq K\), the coordinate formula for
		\(B^n\) gives
		\[
		f(v_{m-\ell,j+n})
		=
		(B^nf)(v_{m-\ell,j})
		-
		\sum_{k=1}^{\ell}
		\binom nk f(v_{m-\ell+k,j+n-k}).
		\]
		Using the induction hypothesis,
		\(|(B^nf)(v_{m-\ell,j})|<\delta\), and
		\[
		-\sum_{k=1}^{\ell}
		\frac{(-1)^{\ell-k}}{k!(\ell-k)!}
		=
		\frac{(-1)^\ell}{\ell!},
		\]
		we obtain
		\[
		\begin{aligned}
			\frac1{n^\ell}
			\left|
			f(v_{m-\ell,j+n})
			-b\frac{(-1)^\ell}{\ell!}n^\ell
			\right|
			&\leq
			\frac{\delta}{n^\ell}
			+\delta\sum_{k=1}^{\ell}2^{\ell-k}
			+\frac{|b|\ell}{n}\\
			&\leq
			\delta\sum_{k=1}^{\ell}2^{\ell-k}
			+\frac{m\beta}{n}
			\leq2^\ell\delta.
		\end{aligned}
		\]
		Here we used \(\delta<\beta\) and \(m\beta/n\leq\delta\).
		Thus
		\begin{equation}\label{eq:positive-asymptotic-bilateral}
			\left|
			f(v_{m-\ell,j+n})
			-b\frac{(-1)^\ell}{\ell!}n^\ell
			\right|
			\leq2^\ell\delta n^\ell
		\end{equation}
		for \(0\leq\ell\leq m-1\) and \(-K+\ell\leq j\leq K\).
		
		Now fix \(v_{i,j}\in E\) and put \(\ell=m-i\).
		Since \(K>J+m\), all the coordinates in
		\[
		(B^nf)(v_{i,j-n})
		=
		\sum_{k=0}^{\ell}
		\binom nk f(v_{i+k,j-k})
		\]
		lie in the fixed window. Hence \eqref{eq:window-pointwise}
		yields
		\[
		\left|
		(B^nf)(v_{i,j-n})-a\binom n\ell
		\right|
		\leq
		\delta\sum_{k=0}^{\ell}\binom nk
		\leq m\delta n^\ell.
		\]
		Together with \eqref{eq:positive-asymptotic-bilateral},
		this gives
		\begin{equation}\label{eq:bilateral-lower-bounds}
			\begin{aligned}
				|f(v_{i,j+n})|
				&\geq
				n^\ell\left(\frac{|b|}{\ell!}-2^\ell\delta\right)
				\geq c_1n^\ell,\\
				|(B^nf)(v_{i,j-n})|
				&\geq
				n^\ell\left(\frac{|a|}{\ell!}-m\delta-\frac{|a|}{n}\right)
				\geq c_2n^\ell.
			\end{aligned}
		\end{equation}
		For the last inequalities, we used
		\[
		\frac1{\ell!}\geq\frac1{(m-1)!},
		\qquad
		\frac{|a|}{n}\leq\delta,
		\qquad
		m+1\leq2^m,
		\qquad
		2^m\delta<\min\{c_1,c_2\}.
		\]
		
		Since \(n\geq n_0>K+J\), we have
		\(g_K(v_{i,j+n})=g_K(v_{i,j-n})=0\).
		The norm bounds on \(f-ag_K\) and \(B^nf-bg_K\), together
		with \eqref{eq:bilateral-lower-bounds}, therefore imply
		\[
		\begin{aligned}
			n^\ell\bigl(|\mu_{i,j+n}|+|\mu_{i,j-n}|\bigr)
			&\leq
			\frac{|\mu_{i,j+n}f(v_{i,j+n})|}{c_1}
			+
			\frac{|\mu_{i,j-n}(B^nf)(v_{i,j-n})|}{c_2}\\
			&<
			\rho(c_1^{-1}+c_2^{-1})
			<\varepsilon.
		\end{aligned}
		\]
		Thus \(n\in\mathcal A(E,\varepsilon)\), proving
		\begin{equation}\label{eq:window-inclusion}
			N_B\bigl(B_X(ag_K,\rho),B_X(bg_K,\rho)\bigr)
			\cap[n_0,\infty)
			\subseteq\mathcal A(E,\varepsilon).
		\end{equation}
		
		We now prove \(\mathrm{(ii)}\Rightarrow\mathrm{(iii)}\).
		Suppose that \(B\) is \(\mathcal F\)-transitive. Fix a nonempty
		finite set \(E\subseteq\widetilde V_m\) and \(\varepsilon>0\).
		Apply \eqref{eq:window-inclusion} with \(a=1\) and \(b=2\),
		and denote the resulting constants by \(K,n_0,\rho_0\).
		Since \(B^qg_K\ne2g_K\) for every \(q\geq0\), as seen by
		evaluating at \(v_{m,K}\), we may choose
		\[
		0<\rho<
		\min\left\{
		\rho_0,\,
		\min_{0\leq q<n_0}
		\frac{\|B^qg_K-2g_K\|_X}{1+\|B^q\|}
		\right\}.
		\]
		Put
		\[
		U=B_X(g_K,\rho),
		\qquad
		V=B_X(2g_K,\rho).
		\]
		The choice of \(\rho\) excludes all times \(0\leq q<n_0\),
		so \eqref{eq:window-inclusion} gives
		\[
		N_B(U,V)\subseteq\mathcal A(E,\varepsilon).
		\]
		Since \(N_B(U,V)\in\mathcal F\) and \(\mathcal F\) is upward
		hereditary, \(\mathcal A(E,\varepsilon)\in\mathcal F\).
		Thus \(\mathrm{(iii)}\) holds.
		
		Finally, \(\mathrm{(ii)}\Rightarrow\mathrm{(iv)}\) is immediate.
		To prove \(\mathrm{(iv)}\Rightarrow\mathrm{(iii)}\), assume
		\eqref{eq:tail-condition-bilateral} and suppose that \(B\) is
		topologically \(\mathcal F\)-recurrent. Fix a non-empty finite set
		\(E\subseteq\widetilde V_m\) and \(\varepsilon>0\). Apply
		\eqref{eq:window-inclusion} with \(a=b=1\), and let
		\(K,n_0,\rho\) be the corresponding constants.
		Put
		\[
		U
		=
		B_X(g_K,\rho).
		\]
		Then
		\[
		N_B(U,U)\in\mathcal F.
		\]
		By \eqref{eq:tail-condition-bilateral},
		\[
		N_B(U,U)\cap[n_0,\infty)
		\in\mathcal F.
		\]
		From \eqref{eq:window-inclusion},
		\[
		N_B(U,U)\cap[n_0,\infty)
		\subseteq
		\mathcal A(E,\varepsilon).
		\]
		Hence
		\[
		\mathcal A(E,\varepsilon)\in\mathcal F
		\]
		by the upward heredity of \(\mathcal F\). Thus (iii) holds, and the
		proof is complete.
	\end{proof}

\section{\(\mathcal F\)-transitivity on the infinite lattice graph}

Baranov, Lishanskii and Papathanasiou
\cite[Theorem~2.3]{BLP} obtained a sufficient condition for mixing
and a necessary condition for hypercyclicity on the infinite lattice,
leaving a complete characterization open. We characterize
\(\mathcal F\)-transitivity for radial weights on \(\ell^2\),
thereby partially answering their open problem.

Let \(G_\infty=(V_\infty,E_\infty)\) be the directed lattice graph
with vertex set
\[
V_\infty=\{v_{i,j}:i,j\in\mathbb N_0\}
\]
and edges
\[
v_{i,j}\longrightarrow v_{i+1,j},
\qquad
v_{i,j}\longrightarrow v_{i,j+1},
\qquad i,j\in\mathbb N_0.
\]
A finite portion of \(G_\infty\) is shown in
Figure~\ref{fig:infinite-lattice-graph}.

\begin{figure}[htbp]
	\centering
	\includegraphics[width=0.6\textwidth]
	{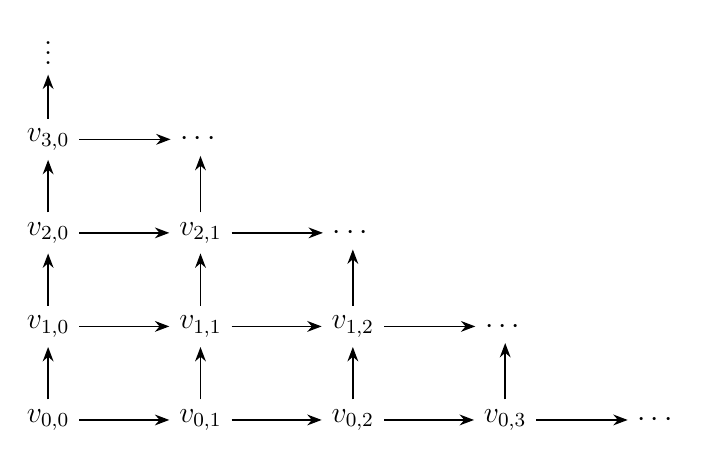}
	\caption{A finite portion of the directed lattice \(G_\infty\).}
	\label{fig:infinite-lattice-graph}
\end{figure}

Let \((\omega_n)_{n\geq0}\) be a sequence of nonzero scalars
and define the \emph{radial weight}
\[
\mu_{i,j}=\omega_{i+j},
\qquad i,j\in\mathbb N_0.
\]
We work on the Hilbert space
\[
X=\ell^2(V_\infty,\mu),
\qquad
\|f\|_X^2
=
\sum_{i,j\geq0}|\omega_{i+j}f(v_{i,j})|^2,
\]
and assume throughout this section that \(B\) is bounded on \(X\).
For \(f\in X\) and \(n\in\mathbb N_0\), its iterates satisfy
\begin{equation}\label{eq:iterate-infinite-lattice}
	(B^nf)(v_{i,j})
	=
	\sum_{r=0}^{n}
	\binom nr f(v_{i+r,j+n-r}),
	\qquad i,j\in\mathbb N_0.
\end{equation}
\subsection{Finite-dimensional reduction and auxiliary estimates}
We first reduce the action of the iterates of \(B\) between diagonal
levels to a family of finite-dimensional  matrices and
derive the estimates needed in the proof of the main theorem.
For \(q\in\mathbb N_0\), let
\[
D_q
=
\{v_{r,q-r}:0\le r\le q\}
\]
and
\[
X_q
=
\{f\in X:\operatorname{supp}(f)\subseteq D_q\}.
\]
Since the sets \(D_q\), \(q\in\mathbb N_0\), form a partition of
\(V_\infty\), the spaces \(X_q\) are pairwise orthogonal and
\[
X
=
\biggl(
\bigoplus_{q=0}^{\infty}X_q
\biggr)_{\ell^2}.
\]

For each \(q\in\mathbb N_0\), define
\[
J_q:X_q\longrightarrow\mathbb C^{q+1}
\]
by
\[
J_qf
=
\bigl(
\omega_qf(v_{0,q}),
\omega_qf(v_{1,q-1}),
\dots,
\omega_qf(v_{q,0})
\bigr).
\]
Then \(J_q\) is an isometric isomorphism.
	
	We first describe the action of the iterates of \(B\) between two
	diagonal levels.
	
	\begin{lemma}\label{lem:diagonal-transfer}
		For every \(q\in\mathbb N_0\) and \(n\in\mathbb N\),
		\[
		J_q\bigl(B^n|_{X_{q+n}}\bigr)J_{q+n}^{-1}
		=
		\frac{\omega_q}{\omega_{q+n}}P_{q,n},
		\]
		where
		\[
		P_{q,n}:\mathbb C^{q+n+1}\longrightarrow\mathbb C^{q+1}
		\]
		is given by
		\[
		(P_{q,n}x)_r
		=
		\sum_{\ell=0}^{n}\binom n\ell x_{r+\ell},
		\qquad 0\le r\le q.
		\]
		Moreover, \(P_{q,n}\) has full row rank. Consequently,
		\[
		B^n|_{X_{q+n}}:X_{q+n}\longrightarrow X_q
		\]
		is surjective.
	\end{lemma}
	\begin{proof}
		Let \(f\in X_{q+n}\) and \(0\le r\le q\). Then
		\[
		(B^nf)(v_{r,q-r})
		=
		\sum_{\ell=0}^{n}
		\binom n\ell
		f(v_{r+\ell,q+n-r-\ell}).
		\]
		Every vertex occurring in this sum belongs to \(D_{q+n}\). Since
		the weight is radial, all vertices in \(D_q\) have weight
		\(\omega_q\), whereas all vertices in \(D_{q+n}\) have weight
		\(\omega_{q+n}\). It follows that
		\[
		J_qB^nf
		=
		\frac{\omega_q}{\omega_{q+n}}
		P_{q,n}J_{q+n}f,
		\]
		which proves the asserted identity.
		
		The submatrix formed by the first \(q+1\) columns of \(P_{q,n}\)
		is upper triangular, with all diagonal entries equal to
		\(\binom n0=1\). Hence
		\[
		\operatorname{rank}P_{q,n}=q+1.
		\]
		Thus \(P_{q,n}\) has full row rank and is therefore surjective.
		Since \(J_q\) and \(J_{q+n}\) are isometric isomorphisms and
		\(\omega_q/\omega_{q+n}\ne0\), the asserted surjectivity of
		\(B^n|_{X_{q+n}}\) follows.
	\end{proof}
Throughout, \(A^*=\overline{A}^{\,T}\) denotes the conjugate
transpose of a complex matrix \(A\). We shall use two classical summation identities.
Vandermonde's convolution, the binomial form of the
Chu--Vandermonde identity
\cite[Corollary~2.2.3]{AAR}, states that
\begin{equation}\label{eq:vandermonde-convolution}
	\sum_{k\in\mathbb Z}
	\binom uk\binom v{t-k}
	=
	\binom{u+v}{t},
	\qquad u,v\in\mathbb N_0,\quad t\in\mathbb Z,
\end{equation}
where binomial coefficients outside their usual range are
understood to be zero.
Writing
\[
(a)_0=1,
\qquad
(a)_j=a(a+1)\cdots(a+j-1)
\quad (j\geq1),
\]
the Pfaff--Saalsch\"utz summation formula
\cite[Eq.~(2.2.8)]{AAR} takes the form
\begin{equation}\label{eq:pfaff-saalschutz}
	\sum_{j=0}^{N}
	\frac{(-N)_j(a)_j(b)_j}
	{j!(c)_j(1+a+b-c-N)_j}
	=
	\frac{(c-a)_N(c-b)_N}
	{(c)_N(c-a-b)_N},
\end{equation}
where \(N\in\mathbb N_0\), provided all denominators are
nonzero.

\begin{lemma}\label{lem:pascal-gram}
	Let \(q\in\mathbb N_0\) be fixed. For integers \(n\geq q\),
	set
	\[
	G_{q,n}=P_{q,n}P_{q,n}^*,
	\qquad
	A_{q,n}=\binom{2n}{n}^{-1}G_{q,n}.
	\]
	Then
	\begin{equation}\label{eq:gram-binomial}
		(G_{q,n})_{r,s}
		=
		\binom{2n}{n+r-s},
		\qquad 0\leq r,s\leq q.
	\end{equation}
	Moreover,
	\begin{equation}\label{eq:LDL-factorization}
		A_{q,n}
		=
		L_{q,n}
		\operatorname{diag}
		\bigl(d_0^{(n)},\ldots,d_q^{(n)}\bigr)
		L_{q,n}^*,
	\end{equation}
	where \(L_{q,n}=(\ell_{r,j}^{(n)})_{0\leq r,j\leq q}\)
	is unit lower triangular with
	\[
	\ell_{r,j}^{(n)}
	=
	\binom rj
	\frac{n!(n+j)!}{(n-r+j)!(n+r)!},
	\qquad 0\leq j\leq r\leq q,
	\]
	and
	\begin{equation}\label{eq:def-djn}
		d_j^{(n)}
		=
		j!\prod_{k=1}^{j}
		\frac{2n+k}{(n+k)^2},
		\qquad 0\leq j\leq q,
	\end{equation}
	with the empty product understood to be \(1\).
\end{lemma}

\begin{proof}
	By the definition of \(P_{q,n}\) and
	\eqref{eq:vandermonde-convolution},
	\[
	\begin{aligned}
		(G_{q,n})_{r,s}
		=
		\sum_{t=0}^{q+n}
		\binom n{t-r}\binom n{t-s}=
		\sum_{k\in\mathbb Z}
		\binom nk\binom n{n+r-s-k}
		=
		\binom{2n}{n+r-s}.
	\end{aligned}
	\]
	Here we used \(k=t-s\) and the symmetry of the binomial
	coefficients. This proves \eqref{eq:gram-binomial} and gives
	\begin{equation}\label{4.2.1}
		(A_{q,n})_{r,s}
		=
		\frac{(n!)^2}
		{(n+r-s)!(n-r+s)!}.
	\end{equation}
	
	Put
	\[
	D_{q,n}
	=
	\operatorname{diag}
	\bigl(d_0^{(n)},\ldots,d_q^{(n)}\bigr).
	\]
	Since both \(A_{q,n}\) and
	\(L_{q,n}D_{q,n}L_{q,n}^*\) are real symmetric,
	it suffices to consider \(0\leq s\leq r\leq q\).
	The identities
	\[
	\frac{\ell_{r,j}^{(n)}}{\ell_{r,0}^{(n)}}
	=
	\frac{(-1)^j(-r)_j(n+1)_j}
	{j!(n-r+1)_j},
	\qquad
	d_j^{(n)}
	=
	\frac{j!(2n+1)_j}{(n+1)_j^2}
	\]
	yield
	\[
	\sum_{j=0}^{s}
	\ell_{r,j}^{(n)}d_j^{(n)}\ell_{s,j}^{(n)}
	=
	\ell_{r,0}^{(n)}\ell_{s,0}^{(n)}
	\sum_{j=0}^{s}
	\frac{(-r)_j(-s)_j(2n+1)_j}
	{j!(n-r+1)_j(n-s+1)_j}.
	\]
	
	Apply \eqref{eq:pfaff-saalschutz} with
	\[
	N=s,\qquad
	a=-r,\qquad
	b=2n+1,\qquad
	c=n-r+1.
	\]
	Then \(1+a+b-c-N=n-s+1\), and all denominators are
	nonzero because \(s\leq r\leq n\). Hence
	\[
	\sum_{j=0}^{s}
	\frac{(-r)_j(-s)_j(2n+1)_j}
	{j!(n-r+1)_j(n-s+1)_j}
	=
	\frac{(n+1)_s(-n-r)_s}
	{(n-r+1)_s(-n)_s}.
	\]
	Using
	\[
	\ell_{r,0}^{(n)}\ell_{s,0}^{(n)}
	=
	\frac{(n!)^4}
	{(n-r)!(n+r)!(n-s)!(n+s)!}
	\]
	and rewriting the rising factorials as factorial ratios,
	we obtain
	\[
	\begin{aligned}
		\bigl(L_{q,n}D_{q,n}L_{q,n}^*\bigr)_{r,s}
		&=
		\ell_{r,0}^{(n)}\ell_{s,0}^{(n)}
		\frac{(n+1)_s(-n-r)_s}
		{(n-r+1)_s(-n)_s}\\
		&=
		\frac{(n!)^2}
		{(n+r-s)!(n-r+s)!}
		=
		(A_{q,n})_{r,s},
	\end{aligned}
	\]
	where the last equality follows from \eqref{4.2.1}.
	This proves \eqref{eq:LDL-factorization}.
\end{proof}

Before stating the asymptotic estimate, we clarify our notation.
Let
\[
T:\mathbb C^N\longrightarrow\mathbb C^m,
\qquad N\geq m,
\]
be a matrix with full row rank. Throughout, matrix norms are
the operator norms induced by the Euclidean norms; thus
\[
\|T\|_{2\to2}
=
\sup_{\substack{x\in\mathbb C^N\\\|x\|_2=1}}
\|Tx\|_2,
\]
where \(\|\cdot\|_2\) denotes the Euclidean norm.

We denote by \(s_{\min}(T)\) its
\emph{smallest positive singular value}. Equivalently
(see, for instance, \cite[Sections~2.6 and~4.2]{HJ}),
\begin{equation}\label{4.3.1}
	s_{\min}(T)
	=
	\lambda_{\min}(TT^*)^{1/2}
	=
	\min_{\substack{y\in\mathbb C^m\\\|y\|_2=1}}
	\|T^*y\|_2.
\end{equation}
By Lemma~\ref{lem:diagonal-transfer}, \(P_{q,n}\) has full row
rank, so this notation applies to \(P_{q,n}\).

\begin{lemma}\label{lem:pascal-singular}
	For every fixed \(q\in\mathbb N_0\),
	\begin{equation}\label{eq:sharp-singular-asymptotic}
		s_{\min}(P_{q,n})
		\sim
		\gamma_q\,2^n n^{-q/2-1/4},
		\qquad n\to\infty,
	\end{equation}
	where
	\[
	\gamma_q
	=
	\Biggl(
	\frac{2^q q!}
	{\sqrt{\pi}\binom{2q}{q}}
	\Biggl)^{1/2}.
	\]
	Consequently, there exist constants \(c_q,C_q>0\) such that
	\[
	c_q2^n(n+1)^{-q/2-1/4}
	\le
	s_{\min}(P_{q,n})
	\le
	C_q2^n(n+1)^{-q/2-1/4},
	\qquad n\ge1.
	\]
\end{lemma}

\begin{proof}
	Fix \(q\in\mathbb N_0\). By
	Lemma~\ref{lem:pascal-gram},
	\[
	P_{q,n}P_{q,n}^*
	=
	\binom{2n}{n}A_{q,n},
	\qquad
	A_{q,n}=L_{q,n}\operatorname{diag}
	\bigl(d_0^{(n)},\ldots,d_q^{(n)}\bigr)L_{q,n}^*.
	\]
	
	We first determine the smallest eigenvalue of \(A_{q,n}\).
	For fixed \(0\le j\le r\le q\),
	\[
	\ell_{r,j}^{(n)}
	=
	\binom rj
	\frac{n!(n+j)!}{(n-r+j)!(n+r)!}
	\longrightarrow
	\binom rj,
	\qquad n\to\infty.
	\]
	Thus
	\[
	L_{q,n}\longrightarrow L_q,
	\qquad n\to\infty,
	\]
	where
	\[
	(L_q)_{r,j}
	=
	\begin{cases}
		\binom rj,&j\le r,\\
		0,&j>r.
	\end{cases}
	\]
	The matrix \(L_q\) is the lower triangular Pascal matrix and
	\begin{equation}\label{4.3.2}
	(L_q^{-1})_{r,j}
	=
	\begin{cases}
		(-1)^{r-j}\binom rj,&j\le r,\\
		0,&j>r.
	\end{cases}
	\end{equation}
	From \eqref{eq:def-djn},
	\[
	d_j^{(n)}
	=
	j!\prod_{k=1}^{j}\frac{2n+k}{(n+k)^2},
	\]
	and hence
	\begin{equation}\label{eq:djn-asymptotic}
		d_j^{(n)}
		\sim
		2^j j!n^{-j},
		\qquad 0\le j\le q.
	\end{equation}
	In particular,
	\[
	\frac{d_q^{(n)}}{d_j^{(n)}}\longrightarrow0,
	\qquad 0\le j<q.
	\]
	
	Since
	\[
	A_{q,n}^{-1}
	=
	L_{q,n}^{-*}D_{q,n}^{-1}L_{q,n}^{-1},
	\]
	we obtain
	\[
	d_q^{(n)}A_{q,n}^{-1}
	=
	L_{q,n}^{-*}
	\operatorname{diag}
	\Biggl(
	\frac{d_q^{(n)}}{d_0^{(n)}},
	\ldots,
	\frac{d_q^{(n)}}{d_{q-1}^{(n)}},
	1
	\Biggr)
	L_{q,n}^{-1}.
	\]
	The dimension being fixed, the preceding observations imply
	convergence in operator norm:
	\[
	d_q^{(n)}A_{q,n}^{-1}
	\longrightarrow
	L_q^{-*}e_qe_q^*L_q^{-1},
	\]
	where \(e_q\) is the last standard basis vector of
	\(\mathbb C^{q+1}\). Therefore,
	\begin{equation}\label{eq:inverse-limit}
		d_q^{(n)}\|A_{q,n}^{-1}\|_{2\to2}
		\longrightarrow
		\|L_q^{-*}e_q\|_2^2.
	\end{equation}
	
	By \eqref{4.3.2},
	\[
	L_q^{-*}e_q
	=
	\bigl(
	(-1)^q\binom q0,
	(-1)^{q-1}\binom q1,
	\ldots,
	\binom qq
	\bigr)^T.
	\]
	Hence, by \eqref{eq:vandermonde-convolution},
	\[
	\|L_q^{-*}e_q\|_2^2
	=
	\sum_{j=0}^{q}\binom qj^2
	=
	\binom{2q}{q}.
	\]
	Since
	\[
	\lambda_{\min}(A_{q,n})
	=
	\|A_{q,n}^{-1}\|_{2\to2}^{-1},
	\]
	it follows from \eqref{eq:inverse-limit} and
	\eqref{eq:djn-asymptotic} that
	\begin{equation}\label{eq:A-smallest-eigenvalue}
		\lambda_{\min}(A_{q,n})
		\sim
		\frac{2^q q!}{\binom{2q}{q}}\,n^{-q}.
	\end{equation}
	
	Finally, Stirling's formula yields
	\[
	\binom{2n}{n}
	\sim
	\frac{4^n}{\sqrt{\pi n}}.
	\]
	Combining this with
\eqref{4.3.1} and
	\eqref{eq:A-smallest-eigenvalue}, we obtain
	\[
	\begin{aligned}
		s_{\min}(P_{q,n})^2
		=
		\binom{2n}{n}\lambda_{\min}(A_{q,n})\sim
		\frac{2^q q!}
		{\sqrt{\pi}\binom{2q}{q}}
		\,4^n n^{-q-1/2}.
	\end{aligned}
	\]
	Taking square roots proves
	\eqref{eq:sharp-singular-asymptotic}.
	
	Since \(n\) and \(n+1\) are equivalent as \(n\to\infty\), the
	two-sided estimate in the statement follows for all sufficiently
	large \(n\). The finitely many remaining values of \(n\) can be
	absorbed into the constants \(c_q\) and \(C_q\), because
	\(P_{q,n}\) has full row rank for every \(n\ge1\).
\end{proof}
	
	We next translate the preceding estimate into an estimate for
	right inverses.
	
	\begin{lemma}\label{lem:right-inverse}
		For \(q\in\mathbb N_0\) and \(n\ge1\), let
		\[
		M_{q,n}
		=
		\frac{\omega_q}{\omega_{q+n}}P_{q,n}
		\]
		and define
		\[
		\kappa_q(n)
		=
		\inf\left\{
		\|R\|_{2\to2}:
		M_{q,n}R=I_{\mathbb C^{q+1}}
		\right\}.
		\]
		Then there exist constants \(a_q,b_q>0\) such that
		\begin{equation}\label{eq:kappa-Phi}
			a_q\Phi_q(n)
			\le
			\kappa_q(n)
			\le
			b_q\Phi_q(n),
			\qquad n\ge1,
		\end{equation}
		where
		\[
		\Phi_q(n)
		=
		\frac{|\omega_{q+n}|}{|\omega_q|}
		\frac{(n+1)^{q/2+1/4}}{2^n}.
		\]
		Moreover, for every \(Q\in\mathbb N_0\), there exists \(C_Q>0\)
		such that
		\begin{equation}\label{eq:Phi-comparison}
			\Phi_q(n)\le C_Q\Phi_Q(n),
			\qquad
			0\le q\le Q,\quad n\ge1.
		\end{equation}
	\end{lemma}
	
	\begin{proof}
	
		By Lemma~\ref{lem:diagonal-transfer}, \(M_{q,n}\) is surjective.
		For a surjective operator between finite-dimensional Hilbert spaces,
		the minimum norm of a right inverse equals the reciprocal of its
		smallest singular value. Indeed, the Moore--Penrose inverse
		\[
		M_{q,n}^{\dagger}
		=
		M_{q,n}^*(M_{q,n}M_{q,n}^*)^{-1}
		\]
		is a right inverse of \(M_{q,n}\), since
		\[
		M_{q,n}M_{q,n}^{\dagger}
		=
		I_{\mathbb C^{q+1}}.
		\]
	Moreover,
	\[
	\|M_{q,n}^{\dagger}\|_{2\to2}
	=
	\|(M_{q,n}M_{q,n}^*)^{-1}\|_{2\to2}^{1/2}
	=
	\lambda_{\min}(M_{q,n}M_{q,n}^*)^{-1/2}
	=
	\frac{1}{s_{\min}(M_{q,n})}.
	\]
		
		Conversely, let \(R\) be any right inverse of \(M_{q,n}\).
		Choose \(y\in\mathbb C^{q+1}\) such that
		\[
		\|y\|_2=1
		\qquad\text{and}\qquad
		\|M_{q,n}^*y\|_2=s_{\min}(M_{q,n}).
		\]
		Since \(M_{q,n}R=I_{\mathbb C^{q+1}}\), taking adjoints gives
		\(
		R^*M_{q,n}^*=I_{\mathbb C^{q+1}}.
		\)
		Consequently,
		\[
		1
		=
		\|y\|_2
		=
		\|R^*M_{q,n}^*y\|_2
		\le
		\|R\|_{2\to2}\|M_{q,n}^*y\|_2
		=
		\|R\|_{2\to2}s_{\min}(M_{q,n}),
		\]
		and hence
		\[
		\|R\|_{2\to2}
		\ge
		\frac{1}{s_{\min}(M_{q,n})}.
		\]
		It follows that
		\[
		\kappa_q(n)
		=
		\frac{1}{s_{\min}(M_{q,n})}.
		\]
		
		Since
		\[
		M_{q,n}
		=
		\frac{\omega_q}{\omega_{q+n}}P_{q,n},
		\]
		we have
		\[
		s_{\min}(M_{q,n})
		=
		\frac{|\omega_q|}{|\omega_{q+n}|}
		s_{\min}(P_{q,n}).
		\]
		Therefore,
		\[
		\kappa_q(n)
		=
		\frac{|\omega_{q+n}|}{|\omega_q|}
		\frac{1}{s_{\min}(P_{q,n})}.
		\]
		Hence \eqref{eq:kappa-Phi} follows from
		Lemma~\ref{lem:pascal-singular}.
		
		Let
		\[
		C_\omega
		=
		\sup_{k\ge0}
		\frac{|\omega_k|}{|\omega_{k+1}|}.
		\]
		The boundedness of \(B\) guarantees \(C_\omega<\infty\). For
		\(0\le q\le Q\),
		\[
		|\omega_{q+n}|
		\le
		C_\omega^{Q-q}|\omega_{Q+n}|.
		\]
		Therefore
		\[
		\begin{aligned}
			\frac{\Phi_q(n)}{\Phi_Q(n)}
			=
			\frac{|\omega_{q+n}|}{|\omega_{Q+n}|}
			\frac{|\omega_Q|}{|\omega_q|}
			(n+1)^{-(Q-q)/2}\le
			C_\omega^{Q-q}
			\frac{|\omega_Q|}{|\omega_q|}.
		\end{aligned}
		\]
		Since \(q\in\{0,\ldots,Q\}\), the right-hand side is bounded by a
		constant depending only on \(Q\), which proves
		\eqref{eq:Phi-comparison}.
	\end{proof}
	
	We also need the following consequence of weak mixing.
	
	\begin{lemma}\label{lem:finite-sum-F}
		If \(T\) is \(\widetilde{\mathcal F}\)-transitive, then
		\(T^{\oplus d}\) is \(\mathcal F\)-transitive for every
		\(d\in\mathbb N\).
	\end{lemma}
	\begin{proof}
		Since \(\widetilde{\mathcal F}\subseteq\mathcal F\), \(T\) is
		\(\mathcal F\)-transitive. Moreover, \(T\) is weakly mixing
		by \cite[Lemma~2.3]{BMPP}.
		
		Fix \(d\in\mathbb N\) and non-empty open sets
		\(U_1,\ldots,U_d,V_1,\ldots,V_d\subseteq X\).
		Set
		\[
		k_1=0,
		\qquad
		U^{(1)}=U_1,
		\qquad
		V^{(1)}=V_1.
		\]
		Since \(T\oplus T\) is topologically transitive, for
		\(j=2,\ldots,d\) we can recursively choose
		\(k_j\in\mathbb N_0\) such that
		\[
		U^{(j)}
		=
		U^{(j-1)}\cap T^{-k_j}(U_j),
		\qquad
		V^{(j)}
		=
		V^{(j-1)}\cap T^{-k_j}(V_j)
		\]
		are both non-empty open sets. Put
		\[
		U=U^{(d)}=\bigcap_{j=1}^{d}T^{-k_j}(U_j),
		\qquad
		V=V^{(d)}=\bigcap_{j=1}^{d}T^{-k_j}(V_j).
		\]
		
		If \(n\in N_T(U,V)\), choose \(x\in U\) such that
		\(T^nx\in V\). For every \(1\leq j\leq d\), we then have
		\[
		T^{k_j}x\in U_j,
		\qquad
		T^n(T^{k_j}x)=T^{k_j}(T^nx)\in V_j.
		\]
		Hence
		\[
		N_T(U,V)
		\subseteq
		\bigcap_{j=1}^{d}N_T(U_j,V_j).
		\]
		Since \(T\) is \(\mathcal F\)-transitive,
		\(N_T(U,V)\in\mathcal F\). The upward heredity of
		\(\mathcal F\) therefore gives
		\[
		N_{T^{\oplus d}}
		\biggl(
		\prod_{j=1}^{d}U_j,\prod_{j=1}^{d}V_j
		\biggr)
		=
		\bigcap_{j=1}^{d}N_T(U_j,V_j)
		\in\mathcal F.
		\]
		As products of non-empty open sets form a base for the topology
		of \(X^d\), \(T^{\oplus d}\) is \(\mathcal F\)-transitive.
	\end{proof}
	
\subsection{Dynamical characterizations and applications}

We now apply the preceding estimates to characterize
\(\mathcal F\)-transitivity for radial weights. Under the isometries
\(J_q\), the restriction of \(B^n\) from \(X_{q+n}\) to \(X_q\) is
represented by the matrix \(M_{q,n}\). By
Lemma~\ref{lem:right-inverse}, the minimum norm of a right inverse of
\(M_{q,n}\) is comparable to \(\Phi_q(n)\), which makes it possible to
express the dynamical properties of \(B\) entirely in terms of the
quantities \(\Phi_q(n)\).

Moreover, a right inverse \(R_{q,n}\) of \(M_{q,n}\) induces the right
inverse
\begin{equation}\label{4.4.1}
	\widetilde R_{q,n}
	=
	J_{q+n}^{-1}R_{q,n}J_q
\end{equation}
of \(B^n|_{X_{q+n}}\), as summarized in the following diagram.
\begin{figure}[H]
	\centering
	\includegraphics[width=0.62\textwidth]
	{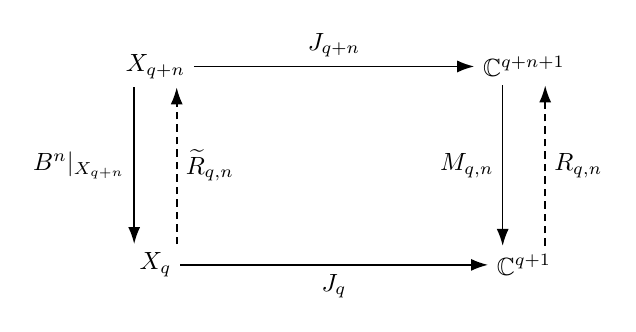}
	\caption{The diagonal-level representation of \(B^n\) and the
		corresponding right inverses.}
	\label{fig:diagonal-level-diagram}
\end{figure}
	\begin{theorem}\label{thm:radial-F-transitivity}
		Let \((\omega_n)_{n\geq0}\) be a sequence of nonzero scalars,
		and put
		\[
		\mu_{i,j}=\omega_{i+j},
		\qquad i,j\geq0.
		\]
		Consider
		\[
		X=\ell^2(V_\infty,\mu),
		\qquad
		\|f\|_X^2
		=
		\sum_{i,j\geq0}
		|\omega_{i+j}f(v_{i,j})|^2.
		\]
		Suppose that \(B\) is bounded on \(X\).
		Let \(\mathcal F\) be a Furstenberg family on
		\(\mathbb N_0\), and let \(\widetilde{\mathcal F}\) be its associated
		enlarged family. Then the following assertions are equivalent:
		\begin{enumerate}
			\item[\textnormal{(i)}]
			\(B\) is \(\widetilde{\mathcal F}\)-transitive;
			
			\item[\textnormal{(ii)}]
			\(B\) is \(\mathcal F\)-transitive;
			
			\item[\textnormal{(iii)}]
			For every \(q\in\mathbb N_0\) and \(\varepsilon>0\),
			\[
			\left\{
			n\in\mathbb N:
			\frac{|\omega_{q+n}|}{|\omega_q|}
			\frac{(n+1)^{q/2+1/4}}{2^n}
			<\varepsilon
			\right\}
			\in\mathcal F.
			\]
		\end{enumerate}
		If, in addition,
		\begin{equation}\label{eq:tail-stable-F}
			A\in\mathcal F
			\quad\Longrightarrow\quad
			A\cap[N,\infty)\in\mathcal F
			\qquad(N\in\mathbb N),
		\end{equation}
		then the preceding assertions are also equivalent to
		\begin{enumerate}
			\item[\textnormal{(iv)}]
			\(B\) is topologically \(\mathcal F\)-recurrent.
		\end{enumerate}
	\end{theorem}
	
	\begin{proof}
		We first prove
		\(\mathrm{(i)}\Longleftrightarrow\mathrm{(ii)}\).
		Since \(\widetilde{\mathcal F}\subseteq\mathcal F\),
		\(\mathrm{(i)}\Rightarrow\mathrm{(ii)}\) is immediate.
		
	Conversely, suppose that \(B\) is \(\mathcal F\)-transitive.
	Then \(B\) is topologically transitive and hence hypercyclic.
	Put
	\[
	\mathcal H
	=
	\ell^2((X_k)_k,\mathbb N_0)
	=
	\left\{
	(x_k)_{k\geq0}:
	x_k\in X_k,\quad
	\sum_{k=0}^{\infty}\|x_k\|_X^2<\infty
	\right\},
	\]
	equipped with the norm
	\[
	\|(x_k)_{k\geq0}\|_{\mathcal H}
	=
	\biggl(\sum_{k=0}^{\infty}\|x_k\|_X^2\biggr)^{1/2}.
	\]
	Define
	\[
	\mathcal U:X\longrightarrow\mathcal H,
	\qquad
	\mathcal Uf=(f\mathbf1_{D_k})_{k\geq0}.
	\]
	The orthogonality of the diagonal levels gives
	\[
	\|\mathcal Uf\|_{\mathcal H}^2
	=
	\sum_{k=0}^{\infty}\|f\mathbf1_{D_k}\|_X^2
	=
	\|f\|_X^2.
	\]
	Moreover, \(\mathcal U\) is onto, with inverse
	\[
	\mathcal U^{-1}(x_k)_{k\geq0}
	=
	\sum_{k=0}^{\infty}x_k,
	\]
	where the series converges in \(X\). Thus \(\mathcal U\) is
	an isometric isomorphism.
	
	For \(k\geq1\), put
	\[
	T_k=B|_{X_k}:X_k\longrightarrow X_{k-1}.
	\]
	Since
	\[
	\sup_{k\geq1}\|T_k\|\leq\|B\|<\infty,
	\]
	the associated unilateral operator-valued backward shift
	\[
	B_{(T_k)}:\mathcal H\longrightarrow\mathcal H,
	\qquad
	B_{(T_k)}(x_0,x_1,\ldots)
	=
	(T_1x_1,T_2x_2,\ldots)
	\]
	is bounded. Since \(B(X_0)=\{0\}\) and
	\(B(X_k)\subseteq X_{k-1}\) for \(k\geq1\), we have
	\[
	\mathcal U B=B_{(T_k)}\mathcal U,
	\qquad\text{equivalently}\qquad
	B_{(T_k)}=\mathcal U B\mathcal U^{-1}.
	\]
	Hence \(B_{(T_k)}\) is hypercyclic. By
	\cite[Proposition~2.2]{MP}, it is weakly mixing, and the
	conjugacy above implies that \(B\) is also weakly mixing.
	Since \(B\) is \(\mathcal F\)-transitive,
	\cite[Lemma~2.3]{BMPP} yields that \(B\) is
	\(\widetilde{\mathcal F}\)-transitive. This proves
	\(\mathrm{(i)}\Longleftrightarrow\mathrm{(ii)}\).
		
		We next prove
		\(\mathrm{(ii)}\Rightarrow\mathrm{(iii)}\).
		Assume \(\mathrm{(ii)}\), and fix
		\(q\in\mathbb N_0\) and \(\varepsilon>0\).
		Let \(a_q>0\) be the constant in \eqref{eq:kappa-Phi}.
		For the standard basis vectors \(e_0,\ldots,e_q\) of
		\(\mathbb C^{q+1}\), put
		\[
		u_r=J_q^{-1}e_r,
		\qquad 0\leq r\leq q.
		\]
		Then \(\|u_r\|_X=1\).
		Choose \(\alpha,\beta>0\) such that
		\begin{equation}\label{eq:alpha-beta}
			\sqrt{q+1}\,\alpha<\frac12,
			\qquad
			2\sqrt{q+1}\,\beta<a_q\varepsilon,
			\qquad
			\alpha+\beta<1.
		\end{equation}
		
		By the equivalence just established, \(B\) is
		\(\widetilde{\mathcal F}\)-transitive. Hence
		Lemma~\ref{lem:finite-sum-F} implies that
		\(B^{\oplus(q+1)}\) is \(\mathcal F\)-transitive, and therefore
		\[
		C=
		\bigcap_{r=0}^{q}
		N_B\bigl(B_X(0,\beta),B_X(u_r,\alpha)\bigr)
		\in\mathcal F.
		\]
		Since \(\alpha+\beta<\|u_r\|_X=1\), the corresponding
		balls are disjoint, so \(0\notin C\).
		
		Fix \(n\in C\). For each \(0\leq r\leq q\), choose
		\(x_{r,n}\in X\) such that
		\begin{equation}\label{eq:x-approximation}
			\|x_{r,n}\|_X<\beta,
			\qquad
			\|B^nx_{r,n}-u_r\|_X<\alpha.
		\end{equation}
		Let \(P_k^X\) denote the orthogonal projection onto \(X_k\).
		Since
		\[
		B^n(X_k)=\{0\}\quad(0\leq k<n),
		\qquad
		B^n(X_k)\subseteq X_{k-n}\quad(k\geq n),
		\]
		the orthogonal decomposition of \(X\) and the boundedness
		of \(B^n\) give
		\[
		P_q^XB^n=B^nP_{q+n}^X.
		\]
		Put \(z_{r,n}=P_{q+n}^Xx_{r,n}\).
		As \(P_k^X\) is contractive and \(u_r\in X_q\),
		\eqref{eq:x-approximation} and
		Lemma~\ref{lem:diagonal-transfer} yield
		\begin{equation}\label{eq:matrix-approximation}
			\|J_{q+n}z_{r,n}\|_2<\beta,
			\qquad
			\|M_{q,n}J_{q+n}z_{r,n}-e_r\|_2<\alpha.
		\end{equation}
		
		Define the linear map
		\[
		S_{q,n}:\mathbb C^{q+1}\longrightarrow\mathbb C^{q+n+1},
		\qquad
		S_{q,n}e_r=J_{q+n}z_{r,n}.
		\]
		For every \(\xi=(\xi_0,\ldots,\xi_q)\in\mathbb C^{q+1}\),
		\eqref{eq:matrix-approximation} gives
		\[
		\begin{aligned}
			\|S_{q,n}\xi\|_2
			&\leq
			\beta\sum_{r=0}^{q}|\xi_r|
			\leq
			\sqrt{q+1}\,\beta\|\xi\|_2,\\
			\|(M_{q,n}S_{q,n}-I)\xi\|_2
			&\leq
			\alpha\sum_{r=0}^{q}|\xi_r|
			\leq
			\sqrt{q+1}\,\alpha\|\xi\|_2.
		\end{aligned}
		\]
		Consequently,
		\begin{equation}\label{eq:S-bound}
			\|S_{q,n}\|_{2\to2}\leq\sqrt{q+1}\,\beta,
		\end{equation}
		and
		\begin{equation}\label{eq:MS-bound}
			\|M_{q,n}S_{q,n}-I\|_{2\to2}
			\leq\sqrt{q+1}\,\alpha<\frac12.
		\end{equation}
		Thus \(M_{q,n}S_{q,n}\) is invertible, with
		\[
		\|(M_{q,n}S_{q,n})^{-1}\|_{2\to2}\leq2.
		\]
		It follows that
		\[
		R_{q,n}
		=
		S_{q,n}(M_{q,n}S_{q,n})^{-1}
		\]
		is a right inverse of \(M_{q,n}\), and
		\[
		\kappa_q(n)
		\leq\|R_{q,n}\|
		\leq2\sqrt{q+1}\,\beta
		<a_q\varepsilon.
		\]
		By \eqref{eq:kappa-Phi},
		\[
		\Phi_q(n)\leq\frac{\kappa_q(n)}{a_q}<\varepsilon.
		\]
		Hence
		\[
		C\subseteq
		\{n\in\mathbb N:\Phi_q(n)<\varepsilon\}.
		\]
		Since \(C\in\mathcal F\) and \(\mathcal F\) is upward
		hereditary, \(\mathrm{(iii)}\) follows.
		
		We now prove
		\(\mathrm{(iii)}\Rightarrow\mathrm{(ii)}\).
		Assume \(\mathrm{(iii)}\), and let \(U,V\subseteq X\)
		be nonempty open sets. Choose finitely supported vectors
		\[
		f\in U,
		\qquad
		0\ne g\in V,
		\]
		and \(\rho>0\) such that
		\begin{equation}\label{eq:balls-UV}
			B_X(f,\rho)\subseteq U,
			\qquad
			B_X(g,\rho)\subseteq V.
		\end{equation}
		Write
		\[
		g=\sum_{q=0}^{Q}g_q,
		\qquad g_q\in X_q.
		\]
		Since \(f\) has finite support, choose \(n_0\geq2\) such that
		\begin{equation}\label{eq:f-killed}
			B^nf=0,
			\qquad n\geq n_0.
		\end{equation}
		
		Let \(C_Q\) be the constant in \eqref{eq:Phi-comparison},
		and put
		\[
		D_Q=C_Q\max_{0\leq q\leq Q}b_q>0,
		\]
		where \(b_q\) is as in \eqref{eq:kappa-Phi}.
		Then
		\begin{equation}\label{eq:kappa-highest-level}
			\kappa_q(n)
			\leq b_q\Phi_q(n)
			\leq D_Q\Phi_Q(n),
			\qquad 0\leq q\leq Q,\quad n\geq1.
		\end{equation}
		All these constants are independent of \(n\).
		Since \(\Phi_Q(k)>0\) for every \(k\geq1\), choose
		\begin{equation}\label{eq:epsilon-small}
			0<\varepsilon<
			\min\left\{
			\frac{\rho}{D_Q\|g\|_X},
			\min_{1\leq k<n_0}\Phi_Q(k)
			\right\}.
		\end{equation}
		By \(\mathrm{(iii)}\),
		\[
		A=
		\{n\in\mathbb N:\Phi_Q(n)<\varepsilon\}
		\in\mathcal F.
		\]
		The second bound in \eqref{eq:epsilon-small} excludes every
		\(1\leq n<n_0\), so
		\begin{equation}\label{eq:remove-small-n}
			A\subseteq[n_0,\infty).
		\end{equation}
		
		Fix \(n\in A\).
		For \(0\leq q\leq Q\), choose a minimum-norm right inverse
		\(R_{q,n}\) of \(M_{q,n}\), as in
		Lemma~\ref{lem:right-inverse}. Then
		\[
		M_{q,n}R_{q,n}=I,
		\qquad
		\|R_{q,n}\|_{2\to2}
		=
		\kappa_q(n)
		\leq D_Q\Phi_Q(n)
		<D_Q\varepsilon.
		\]
		Using \eqref{4.4.1}, define
		\begin{equation}\label{eq:def-yn}
			y_n
			=
			\sum_{q=0}^{Q}\widetilde R_{q,n}g_q
			=
			\sum_{q=0}^{Q}
			J_{q+n}^{-1}R_{q,n}J_qg_q.
		\end{equation}
		The summands belong to the pairwise orthogonal spaces
		\(X_{q+n}\). Hence
		\[
		\begin{aligned}
			\|y_n\|_X^2
			=
			\sum_{q=0}^{Q}
			\|R_{q,n}J_qg_q\|_2^2\leq
			D_Q^2\varepsilon^2
			\sum_{q=0}^{Q}\|g_q\|_X^2=
			D_Q^2\varepsilon^2\|g\|_X^2
			<\rho^2.
		\end{aligned}
		\]
		Moreover, by the diagonal transfer identity,
		\[
		B^ny_n
		=
		\sum_{q=0}^{Q}
		J_q^{-1}M_{q,n}R_{q,n}J_qg_q
		=
		\sum_{q=0}^{Q}g_q
		=
		g.
		\]
		Thus \(x_n=f+y_n\) satisfies
		\[
		\|x_n-f\|_X<\rho,
		\qquad
		B^nx_n=B^nf+B^ny_n=g,
		\]
		where we used \(n\geq n_0\).
		By \eqref{eq:balls-UV}, \(x_n\in U\) and \(B^nx_n\in V\).
		Therefore
		\[
		A\subseteq N_B(U,V).
		\]
		Since \(A\in\mathcal F\), upward heredity yields
		\(N_B(U,V)\in\mathcal F\). This proves
		\(\mathrm{(iii)}\Rightarrow\mathrm{(ii)}\), and hence
		\(\mathrm{(i)}\)--\(\mathrm{(iii)}\) are equivalent.
		
		Finally, assume \eqref{eq:tail-stable-F}.
		The implication
		\(\mathrm{(ii)}\Rightarrow\mathrm{(iv)}\) is immediate.
		Conversely, suppose that \(B\) is topologically
		\(\mathcal F\)-recurrent.
		Let \(U,V\subseteq X\) be nonempty open sets.
		Choose finitely supported vectors \(u\in U\), \(v\in V\),
		and \(\eta>0\) such that
		\[
		B_X(u,\eta)\subseteq U,
		\qquad
		B_X(v,\eta)\subseteq V.
		\]
		Put
		\[
		h=u-v,
		\qquad
		V_0=B_X(v,\eta).
		\]
		Then \(h+V_0\subseteq U\) and \(V_0\subseteq V\).
		Since \(h\) has finite support, there exists \(N\in\mathbb N\)
		such that \(B^nh=0\) for every \(n\geq N\).
		
		By topological \(\mathcal F\)-recurrence and
		\eqref{eq:tail-stable-F},
		\[
		A=N_B(V_0,V_0)\cap[N,\infty)\in\mathcal F.
		\]
		For every \(n\in A\), choose \(z_n\in V_0\) such that
		\(B^nz_n\in V_0\). Then
		\[
		h+z_n\in U,
		\qquad
		B^n(h+z_n)=B^nz_n\in V_0\subseteq V.
		\]
		Hence \(A\subseteq N_B(U,V)\).
		Upward heredity gives \(N_B(U,V)\in\mathcal F\).
		Thus \(B\) is \(\mathcal F\)-transitive, proving
		\(\mathrm{(iv)}\Rightarrow\mathrm{(ii)}\).
	\end{proof}
	Combining Theorem~\ref{thm:radial-F-transitivity} with
	Remark~\ref{rem:classical-Furstenberg-families}, we obtain the
	following criteria for hypercyclicity, weak mixing and mixing.
	\begin{corollary}\label{cor:radial-hypercyclic-mixing}
		Under the assumptions of
		Theorem~\ref{thm:radial-F-transitivity}, the following assertions
		hold.
		\begin{enumerate}
			\item[\textnormal{(i)}]
			\(B\) is hypercyclic, equivalently weakly mixing, if and only if
			\[
			\liminf_{n\to\infty}
			\frac{|\omega_{q+n}|}{|\omega_q|}
			\frac{(n+1)^{q/2+1/4}}{2^n}
			=
			0
			\qquad\text{for every }q\in\mathbb N_0.
			\]
			
			\item[\textnormal{(ii)}]
			\(B\) is mixing if and only if
			\[
			\lim_{n\to\infty}
			\frac{|\omega_{q+n}|}{|\omega_q|}
			\frac{(n+1)^{q/2+1/4}}{2^n}
			=
			0
			\qquad\text{for every }q\in\mathbb N_0.
			\]
		\end{enumerate}
	\end{corollary}
The following example shows that radial weights with the same
critical exponential growth rate \(2\) can give rise to both
mixing and non-hypercyclic shifts. It also gives a negative
answer to the sufficiency question raised after
\cite[Proposition~5.1]{BLP}, even in the radial
\(\ell^2\)-setting.
	\begin{example}\label{ex:critical-radial-weights}
		Fix \(a>1/4\). For \(r=1,2\), let
		\[
		\mu_{i,j}^{(r)}=\omega_{i+j}^{(r)},
		\qquad i,j\in\mathbb N_0,
		\]
		where
		\[
		\omega_k^{(1)}=2^ke^{-\sqrt{k}},
		\qquad
		\omega_k^{(2)}=2^k(k+1)^{-a}.
		\]
		Let \(B_r\) be the backward shift on
		\[
		X_r=\ell^2(V_\infty,\mu^{(r)}).
		\]
		Both operators are bounded, since
		\[
		\sup_{k\ge0}
		\frac{|\omega_k^{(r)}|}
		{|\omega_{k+1}^{(r)}|}
		<\infty,
		\qquad r=1,2.
		\]
		Moreover, for both weights,
		\[
		\lim_{i+j\to\infty}
		|\mu_{i,j}^{(r)}|^{1/(i+j)}=2,
		\qquad
		\frac{|\mu_{i,j}^{(r)}|}{2^{i+j}}
		\longrightarrow0.
		\]
		Thus neither part of \cite[Theorem~2.3]{BLP} determines the
		dynamics of \(B_r\).
		
		For the first weight, every fixed \(q\in\mathbb N_0\) satisfies
		\[
		\begin{aligned}
			\frac{|\omega_{q+n}^{(1)}|}{|\omega_q^{(1)}|}
			\frac{(n+1)^{q/2+1/4}}{2^n}
			&=
			e^{\sqrt q-\sqrt{q+n}}
			(n+1)^{q/2+1/4}
			\longrightarrow0.
		\end{aligned}
		\]
		Hence Corollary~\ref{cor:radial-hypercyclic-mixing}\textnormal{(ii)}
		shows that \(B_1\) is mixing.
		
		For the second weight,
		\[
		\begin{aligned}
			\frac{|\omega_{q+n}^{(2)}|}{|\omega_q^{(2)}|}
			\frac{(n+1)^{q/2+1/4}}{2^n}
			&=
			(q+1)^a
			\frac{(n+1)^{q/2+1/4}}
			{(q+n+1)^a}\\
			&\sim
			(q+1)^a n^{q/2+1/4-a}.
		\end{aligned}
		\]
		Choose a fixed \(q\in\mathbb N_0\) such that
		\(
		\frac q2+\frac14>a.
		\)
		For this \(q\), the preceding expression tends to infinity.
		Therefore
		Corollary~\ref{cor:radial-hypercyclic-mixing}\textnormal{(i)}
		shows that \(B_2\) is not hypercyclic.
		
		Nevertheless, the necessary condition in
		\cite[Proposition~5.1]{BLP} is satisfied by \(B_2\). Indeed, for
		every fixed \(i,j\in\mathbb N_0\), Vandermonde's identity and
		Stirling's formula give
		\[
		\begin{aligned}
			\sum_{\ell=0}^{n}
			\frac{\binom n\ell^2}
			{|\mu_{i+\ell,j+n-\ell}^{(2)}|^2}
			&=
			\frac{\binom{2n}{n}}
			{|\omega_{i+j+n}^{(2)}|^2}\\
			&=
			\frac{\binom{2n}{n}(i+j+n+1)^{2a}}
			{4^{i+j+n}}\\
			&\sim
			\frac{4^{-(i+j)}}{\sqrt{\pi}}\,
			n^{2a-1/2}
			\longrightarrow\infty.
		\end{aligned}
		\]
		Thus the necessary condition in
		\cite[Proposition~5.1]{BLP} is not sufficient, even for radial
		weights on \(\ell^2(V_\infty,\mu)\).
	\end{example}
\section{Further questions}

We conclude with two questions suggested by our results.

\begin{enumerate}
	\item[\textnormal{(1)}]
	\emph{Other sequence spaces.}
	Can the radial characterizations of \(\mathcal F\)-transitivity
	on \(G_\infty\) be extended from \(\ell^2(V_\infty,\mu)\) to
	\(\ell^p(V_\infty,\mu)\), \(1\leq p<\infty\), and
	\(c_0(V_\infty,\mu)\)?
	The main difficulty is to obtain sharp right-inverse estimates
	in the corresponding norms, for which the Hilbert space
	singular-value argument no longer applies directly.
	
	\item[\textnormal{(2)}]
	\emph{Non-radial weights.}
	Can one characterize \(\mathcal F\)-transitivity of bounded
	backward shifts on \(\ell^2(V_\infty,\mu)\) by explicit
	conditions on a general, non-radial weight \(\mu\)?
	The difficulty lies in controlling the variation of the weights
	within each diagonal level, since the transfer matrices are
	no longer scalar multiples of the Pascal-type matrices.
\end{enumerate}

\backmatter

\section*{Declarations}

\bmhead{Conflict of Interest}
The authors declare no competing interests with the journal's editors, reviewers, or readers.

\bmhead{Funding}
This work was supported by  the National Natural Science Foundation of
China (Grant Nos. 12571088) and  the Natural Science Foundation of Henan  (No.252300420341).

\bmhead{Data Availability Statement}
No datasets were generated or analyzed in this theoretical study.

\end{document}